\documentclass[oneside]{book}
\usepackage{amsthm}
\usepackage{tikz}
\usetikzlibrary{positioning, arrows, shapes}
\usetikzlibrary{decorations.markings}
\usepackage{appendix}
\tikzstyle arrowstyle=[scale=1.5]
\tikzstyle directed=[postaction={decorate,decoration={markings,
    mark=at position .5 with {\arrow[arrowstyle]{stealth}}}}]
\tikzstyle reverse directed=[postaction={decorate,decoration={markings,
    mark=at position .5 with {\arrowreversed[arrowstyle]{stealth};}}}]
\usepackage{graphicx}
\usepackage{amsmath, amssymb}

\newcommand{\ddt}{\frac{\partial}{\partial t}}

\newcommand{\ddbar}{\sqrt{-1} \partial \overline{\partial}}
\newcommand{\Ric}{\mathrm{Ric}}
\newcommand{\ov}[1]{\overline{#1}}

\newcommand{\tr}[2]{\textrm{tr}_{#1}{#2}}

\newcommand{\ve}{\varepsilon}

\newcommand{\oke}{\omega_{\textrm{KE}}}
\newcommand{\Kod}{\textrm{Kod}}
\newcommand{\dtu}[2]{\frac{\partial \tilde{z}^{#1}}{\partial z^{#2}}}
\newcommand{\dut}[2]{\frac{\partial z^{#1}}{\partial \tilde{z}^{#2}}}
\newcommand{\dt}[1]{\frac{\partial}{\partial \tilde{z}^{#1}}}
\newcommand{\du}[1]{\frac{\partial}{\partial z^{#1}}}

\renewcommand{\le}{\leqslant}
\renewcommand{\ge}{\geqslant}

\theoremstyle{plain}
\newtheorem{thm}{Theorem}[chapter]
\newtheorem{lem}{Lemma}[chapter]
\newtheorem{prop}{Proposition}[chapter]
\newtheorem{cor}{Corollary}[chapter]
\theoremstyle{definition}
\newtheorem{example}{Example}[chapter]
\newtheorem{exer}{Exercise}[chapter]
\theoremstyle{remark}
\newtheorem{rem}{Remark}[chapter]

\numberwithin{section}{chapter}
\numberwithin{equation}{section}
\numberwithin{figure}{chapter}

\begin{document}


\title{The K\"ahler-Ricci flow on compact K\"ahler manifolds}
\author{Ben Weinkove \\ Department of Mathematics, Northwestern University \\2033 Sheridan Road, Evanston IL 60208}
\date{October 2013}

\maketitle

\chapter*{Preface}

These lecture notes are based on five hours of lectures given  at the Park City Math Institute in the summer of 2013.  The notes are intended to be a leisurely introduction to the K\"ahler-Ricci flow on compact K\"ahler manifolds.  They are
 aimed at graduate students who have some background in differential geometry, but do not necessarily have any knowledge of K\"ahler geometry or the Ricci flow.  
  There are exercises throughout the text.  The goal is that by the end, the reader will learn the basic techniques in the K\"ahler-Ricci flow and know enough to be able to explore the current literature.

The material covered by these notes is as follows.  In the first lecture, we give a quick introduction to some of the main definitions and tools of K\"ahler geometry. In Lecture 2, we introduce the K\"ahler-Ricci flow and give some simple examples, before stating, and in Lecture 3 proving, the maximal existence time theorem for the flow. In Lecture 4 we prove long time convergence results  in the cases when the manifold has negative or zero first Chern class.  Finally in Lecture 5 we discuss more recent work on the behavior of the flow on K\"ahler surfaces.  We also  ``go beyond'' the K\"ahler-Ricci flow by discussing a new flow on complex manifolds called the Chern-Ricci flow.

The K\"ahler-Ricci flow started as a small branch of the study of Hamilton's Ricci flow, but by now is itself a vast area of research. As a consequence, we have had to  omit many topics.  For the interested reader seeking more complete expository sources: the chapter \cite{SW4} by Jian Song and the author  contains many of the results of these notes  and much more; the works  \cite{BG, Cao2, Gu} (in the same volume as \cite{SW4}) and the more general survey \cite{PS2} are excellent sources of information.   

The author thanks Matt Gill who was the teaching assistant for this course, for his help in writing the exercises.  In addition, thanks go to  the organizers, Hubert Bray, Greg Galloway, Rafe Mazzeo and Natasa Sesum, of the research program of the 2013 PCMI Summer Session for giving the author the opportunity to participate in this exciting event.  Discussions with researchers and graduate students at the Park City Math Institute were invaluable in shaping the form of these notes.  The author also thanks Valentino Tosatti for some helpful comments on a  previous version of these notes.

The author was supported in part by NSF grant DMS-1332196. 

\chapter{An Introduction to K\"ahler geometry}

In this lecture we introduce the notion of a K\"ahler metric and describe the associated covariant derivatives and curvatures.  We take a somewhat informal approach which emphasizes the minimal definitions and tools needed to carry out computations.  The reader looking for more details may wish to consult \cite{KM} or \cite{GH}, for example.

\section{Complex manifolds}

Let $M$ be a smooth manifold of dimension $2n$.  We say that $M$ is a \emph{complex manifold of complex dimension $n$} if $M$ can be covered by charts $(U,z)$ where $U$ is an open subset of $M$ and $z: U \rightarrow \mathbb{C}^n$ is a homeomorphism onto an open subset $z(U)$ of $\mathbb{C}^n$, with the following property:  if $(\tilde{U}, \tilde{z})$ is another chart with $U \cap \tilde{U}$ nonempty then the \emph{transition maps}
$$\tilde{z} \circ z^{-1} : z(U \cap \tilde{U}) \rightarrow \tilde{z}(U \cap \tilde{U})$$
and
$$z \circ \tilde{z}^{-1} : \tilde{z}(U \cap \tilde{U}) \rightarrow z(U \cap \tilde{U}),$$
are holomorphic.

We write $z=(z^1, \ldots, z^n)$ and $\tilde{z}= (\tilde{z}^1, \ldots, \tilde{z}^n)$.  These are called \emph{complex coordinates}.   We also introduce the real coordinates $(x^1, \ldots, x^n, y^1, \ldots, y^n)$ defined by the usual formula
$$z^i = x^i + \sqrt{-1} y^i.$$
Note that we avoid the notation $i$ for $\sqrt{-1}$ since $i$ is our favorite letter for an index.
We  define operators $\frac{\partial}{\partial z^i}$ and $\frac{\partial}{\partial \ov{z}^i}$ on $\mathbb{C}^n$ by
$$\frac{\partial}{\partial z^i} = \frac{1}{2} \left( \frac{\partial}{\partial x^i} - \sqrt{-1} \frac{\partial}{\partial y^i} \right), \quad \frac{\partial}{\partial \ov{z}^i} =  \frac{1}{2} \left( \frac{\partial}{\partial x^i} + \sqrt{-1} \frac{\partial}{\partial y^i} \right).$$
As the notation suggests, we have for all $i,j=1, \ldots, n$,
$$\frac{\partial}{\partial z^i} (z^j) = \delta_{ij}, \quad \frac{\partial}{\partial z^i} (\ov{z}^j) =0, \quad \frac{\partial}{\partial \ov{z}^i} (z^j) = 0, \quad \frac{\partial}{\partial \ov{z}^i} (\ov{z}^j) =\delta_{ij},$$
where $\delta_{ij}$ is the Kronecker delta symbol.
In particular, a smooth function $f$ on $\mathbb{C}^n$ is holomorphic if and only if 
$$\frac{\partial f}{\partial \ov{z}^i} = 0 \quad \textrm{for } i=1, \ldots, n.$$
Hence the condition that the transition maps be holomorphic can be written as
\begin{equation} \label{transition}
\frac{\partial \tilde{z}^i}{\partial \ov{z}^j} =0 \quad \textrm{and} \quad \frac{\partial z^i}{\partial \ov{\tilde{z}}^j}=0, \quad \textrm{for all } i,j=1, \ldots, n,
\end{equation}
where defined.  Here we are writing $\tilde{z}$ for the map $\tilde{z} \circ z^{-1}$ and $z$ for the map $z \circ \tilde{z}^{-1}$.

Recall that the definition of a smooth manifold requires the existence of coordinate charts whose transition maps are smooth, and this allows for a well-defined notion of a smooth function.  Similarly, on a complex manifold $M$ we can make the following definition:  a smooth function $f$ on $M$ is \emph{holomorphic} if for each coordinate chart $(U, z)$ we have
$$\frac{\partial f}{\partial \ov{z}^i} =0, \quad \textrm{for } i=1, \ldots, n,$$
on $U$.  Of course, we are writing $f$ for $f \circ z^{-1}$.  To see that this is well-defined, suppose that $\tilde{U}$ is an overlapping coordinate chart and compute using the chain rule on $U \cap \tilde{U}$,
$$\frac{\partial f}{\partial \ov{\tilde{z}}^j} = \sum_k \frac{\partial f}{\partial \ov{z}^k} \frac{\partial \ov{z}^k}{\partial \ov{\tilde{z}}^j} + \sum_k \frac{\partial f}{\partial z^k} \frac{\partial z^k}{\partial \ov{\tilde{z}}^j} = 0,$$
as required.
Note that we have used the condition (\ref{transition}) to see that the second term vanishes.

\begin{example}
$\mathbb{C}^n$ is a complex manifold with a single coordinate chart $U=\mathbb{C}^n$ and $z: U \rightarrow \mathbb{C}^n$ the identity map.  Taking the quotient of $\mathbb{C}^n$ by the lattice $\mathbb{Z}^{2n}$, say, gives a compact complex manifold homeomorphic to the torus $T^{2n}$.
\end{example}

\begin{example} \label{examplePn}
Define \emph{complex projective space} $\mathbb{P}^n$ as follows.  As a topological space, $\mathbb{P}^n$ is the quotient space 
$$(\mathbb{C}^{n+1} - \{ 0 \}) / \sim.$$
where $\sim$ is the equivalence relation defined by
$$(Z_0, Z_1, \ldots, Z_n) \sim (\lambda Z_0, \lambda Z_1, \ldots, \lambda Z_n),$$
for $\lambda \in \mathbb{C}^*$.  In other words, it is the space of complex lines through the origin in $\mathbb{C}^{n+1}$.  Define open sets $$U_i = \{ [Z_0, \ldots, Z_n] \in \mathbb{P}^n \ | \ Z_i \neq 0 \}, \quad \textrm{for } i=0, 1, \ldots, n,$$
where we are writing $[Z_0, \ldots, Z_n]$ for the equivalence class of $(Z_0, \ldots, Z_n)$ (the $Z_i$ are called \emph{homogeneous coordinates}).  The $U_i$ cover $\mathbb{P}^n$.  On $U_0$ we define complex coordinates $z^1, \ldots, z^n$ by
$$z^1 = \frac{Z_1}{Z_0}, \ldots, z^n = \frac{Z_n}{Z_0},$$
and similarly for $U_1, \ldots, U_n$.  We leave it to the reader to check that the associated transition maps are holomorphic.
\end{example}

\begin{exer} \label{S2}
Show that $\mathbb{P}^1$ is diffeomorphic to the sphere $S^2$.
\end{exer}

\section{Vector fields, 1-forms, Hermitian metrics and tensors} \label{sectionvf}

Let $M$ be a complex manifold as above. The \emph{complexified tangent space} $(T_pM)^{\mathbb{C}}$ at a point $p$ is given by the span over $\mathbb{C}$ of
$$\frac{\partial}{\partial z^1}, \ldots, \frac{\partial}{\partial z^n}, \frac{\partial}{\partial \ov{z}^1}, \ldots, \frac{\partial}{\partial \ov{z}^n},$$
where we evaluate at the point $p$.
We write $(T_pM)^{\mathbb{C}} = T^{1,0}_pM \oplus T^{0,1}_pM$, where 
$T^{1,0}_pM$ is given by the span of the $\frac{\partial}{\partial z^i}$ and $T^{0,1}_pM$ by the $\frac{\partial}{\partial \ov{z}^i}$.  By the chain rule and the equations (\ref{transition}), this decomposition of $(T_pM)^{\mathbb{C}}$ is independent of choice of complex coordinate chart.
Indeed,
$$\frac{\partial}{\partial z^i} = \sum_k \frac{\partial \tilde{z}^k}{\partial z^i} \frac{\partial}{\partial \tilde{z}^k} +\sum_k \frac{\partial \ov{\tilde{z}}^k}{\partial z^i} \frac{\partial}{\partial \ov{\tilde{z}}^k} =  \sum_k \frac{\partial \tilde{z}^k}{\partial z^i} \frac{\partial}{\partial \tilde{z}^k} \in \textrm{span}\left\{ \frac{\partial}{\partial \tilde{z}^1}, \ldots, \frac{\partial}{\partial \tilde{z}^n} \right\},$$
and similarly for $\frac{\partial}{\partial \ov{z}^i}$.

We define a \emph{$T^{1,0}$ vector field} on $M$ to be a smooth complex-valued vector field $X$ on $M$ with the property that $X_p \in T^{1,0}_pM$ for all $p\in M$.  We write $X$ locally as 
$$X= \sum_i X^i \frac{\partial}{\partial z^i},$$
where the $X^i: U \rightarrow \mathbb{C}$ are smooth functions which satisfy the following \emph{transformation rule}.  If we write $\tilde{X}^i$ for the corresponding functions on $\tilde{U}$ then
$$X^i = \sum_j \tilde{X}^j  \frac{\partial z^i}{\partial \tilde{z}^j} \quad \textrm{on } U \cap \tilde{U}.$$
Any collection of 
 functions $X^i: U \rightarrow \mathbb{C}$ defined on each chart in a cover of $M$, which satisfy the above transformation rule, determine a globally defined vector field $X$. Indeed  one can check that  
$$X^i \frac{\partial}{\partial z^i} = \tilde{X}^i \frac{\partial}{\partial \tilde{z}^i} \quad \textrm{on } U \cap \tilde{U}.$$
Here and henceforth we are using the \emph{summation convention} that we sum over repeated indices from $1$ to $n$  when one index is upper and the other is lower (we regard the index $i$ in $\frac{\partial}{\partial z^i}$ as  a lower index).

\begin{exer} \label{exh}
We define a \emph{holomorphic vector field} on $M$ to be a $T^{1,0}$ vector field $X=X^i \frac{\partial}{\partial z^i}$ such that
$$\frac{\partial X^i}{\partial \ov{z}^j} =0 \quad \textrm{for all } i,j=1, \ldots, n.$$
Show that this condition is well-defined, independent of choice of coordinate chart.
\end{exer}

We also define $T^{0,1}$ vector fields in a similar way.  A $T^{0,1}$ vector field is written locally as $Y= Y^{\ov{j}} \frac{\partial}{\partial \ov{z}^j}$ where the $Y^{\ov{j}}$ transform according to the rule
$$Y^{\ov{j}} = \tilde{Y}^{\ov{\ell}} \ov{ \frac{\partial z^j}{\partial \tilde{z}^{\ell}}}\quad \textrm{on } U \cap \tilde{U},$$
where of course
$$\ov{ \frac{\partial z^j}{\partial \tilde{z}^{\ell}}} =   \frac{\partial \ov{z}^j}{\partial \ov{\tilde{z}}^{\ell}}.$$

We can do the same for the complexified \emph{cotangent} space $(T_p^*M)^{\mathbb{C}}$ which is spanned over $\mathbb{C}$ by the $1$-forms
$$dz^1, \ldots, dz^n, d\ov{z}^1, \ldots, d\ov{z}^n.$$
Here $dz^i = dx^i + \sqrt{-1} dy^i$ and  $d\ov{z}^i = dx^i - \sqrt{-1} dy^i$ are dual to $\frac{\partial}{\partial z^i}$ and $\frac{\partial}{\partial \ov{z}^i}$ respectively.  We have  a decomposition of the complexified cotangent space into $(1,0)$-forms and $(0,1)$-forms, spanned by the $dz^i$ and $d\ov{z}^i$ respectively.  A  $(1,0)$-form $a$ on $M$ is written locally as $a= a_i dz^i$, and a $(0,1)$ form $b$ as $b=b_{\ov{j}} d\ov{z}^j$, where the $a_i$ and $b_{\ov{j}}$ transform by
$$a_i = \tilde{a}_k \frac{\partial \tilde{z}^k}{\partial z^i}, \quad b_{\ov{j}} = \tilde{b}_{\ov{\ell}} \ov{\frac{\partial \tilde{z}^{\ell}}{\partial z^j}} \quad \textrm{on } U \cap \tilde{U}.$$

Finally, define
a \emph{Hermitian metric} $g$ on $M$ to be a Hermitian inner product on the $n$-dimensional complex vector space $T^{1,0}_pM$ for each $p$, which varies smoothly in $p$.   Locally $g$ is given by an $n\times n$ positive definite Hermitian matrix whose $(i,j)$th entry we denote by $g_{i\ov{j}}$, which transforms according to
\begin{equation} \label{gtrans}
g_{i\ov{j}} = \tilde{g}_{k\ov{\ell}} \frac{\partial \tilde{z}^k}{\partial z^i} \ov{\frac{\partial \tilde{z}^{\ell}}{\partial z^j}} \quad \textrm{on } U \cap \tilde{U}.
\end{equation}
Given $T^{1,0}$ vector fields $X=X^i \frac{\partial}{\partial z^i}$ and $Y=Y^i \frac{\partial}{\partial z^i}$ we define their pointwise inner product by
$$\langle X, Y \rangle_g = g_{i\ov{j}}X^i \ov{Y^j},$$
and we write
$$|X|_g = \sqrt{\langle X, X \rangle_g}$$
for the \emph{norm} of $X$ with respect to $g$.

\begin{exer} \label{exinp}
Show that $\langle X, Y\rangle_g$ is well-defined, independent of choice of complex coordinates.
\end{exer}

We can similarly use $g$ to define an inner product on $T^{0,1}$ vectors.

\begin{rem}
A Hermitian metric $g$ defines a Riemannian metric $g_R$, which we can define locally by
$$g_R \left( \frac{\partial}{\partial x^i} , \frac{\partial}{\partial x^j} \right) = 2 \textrm{Re} (g_{i\ov{j}}) = g_R \left( \frac{\partial}{\partial y^i}, \frac{\partial}{\partial y^j} \right), \quad g_R\left( \frac{\partial}{\partial x^i}, \frac{\partial}{\partial y^j} \right) = 2 \textrm{Im}(g_{i\ov{j}}).$$
However, we won't make use of this correspondence.
\end{rem}

Extending all of the above, we can define tensors on a complex manifold with any number of upper or lower indices, barred or unbarred.  For example, the set of locally defined functions  $S^{ik}_{\ov{j}} : U \rightarrow \mathbb{C}$ defines a tensor with two upper unbarred indices and one lower barred index, if it satisfies the transformation rule:
\begin{equation} \label{tensor}
S^{ik}_{\ov{j}} = \tilde{S}^{ab}_{\ov{c}} \frac{\partial z^i}{\partial \tilde{z}^a} \frac{\partial z^k}{\partial \tilde{z}^b} \ov{\frac{\partial \tilde{z}^c}{\partial z^j}} \quad \textrm{on } U \cap \tilde{U}.
\end{equation}
More formally, $S = S^{ik}_{\ov{j}} \frac{\partial}{\partial z^i} \otimes \frac{\partial}{\partial z^k} \otimes d\ov{z}^j$ defines a smooth section of $T^{1,0}M \otimes T^{1,0}M \otimes (T^{0,1})^*M$.  However, in these lecture notes we will stick with more informal language.

The reader will notice that the transformation formulae for any kind of tensor can easily be derived by following the simple rule:  match the indices according to barred/unbarred, upper/lower and $z$ or $\tilde{z}$.  For example, in (\ref{tensor}), the indices $i,k$ on the left hand side are unbarred upper indices with respect to $z$, and since they are ``free indices'', they match with corresponding free unbarred indices $i,k$ with respect to $z$ on the right.  On the other hand, the upper unbarred indices $a, b$ with respect to $\tilde{z}$ on the right are ``summed indices'' and so must match with \emph{lower} unbarred $a,b$ indices with respect to $\tilde{z}$.  

\begin{exer} \label{ex14}
Let $g=(g_{k\ov{\ell}})$ be a Hermitian metric on $M$.    Define $g^{i\ov{j}}$ to be the $(i,j)$th component of the inverse matrix of $(g_{k\ov{\ell}})$.  Show that $g^{i\ov{j}}$ defines a tensor on $M$, which we call $g^{-1}$.
\end{exer}

We define an pointwise inner product on $(1,0)$ forms using $g^{i\ov{j}}$, as follows.  If $a=a_i dz^i$ and $b = b_i dz^i$ then
$$\langle a, b \rangle_g = g^{i\ov{j}} a_i \ov{b_j},$$
and we define the norm of $a$ to be $|a|_g = \sqrt{ \langle a, a \rangle_g}.$  We can similarly define an inner product  for $(0,1)$ forms.

\section{K\"ahler metrics and covariant differentiation}

We say that a Hermitian metric $g=(g_{i\ov{j}})$ is \emph{K\"ahler} if
\begin{equation} \label{kahlercondition}
\partial_k g_{i\ov{j}} = \partial_i g_{k\ov{j}} \quad \textrm{for all } i,j,k=1, \ldots, n.
\end{equation}
Namely,  $\partial_k g_{i\ov{j}}$ is unchanged when we swap the two unbarred indices $k$ and $i$.  Here and henceforth, to simplify notation, we are writing
$$\partial_i = \frac{\partial}{\partial z^i} \quad \textrm{and} \quad \partial_{\ov{j}} = \frac{\partial}{\partial \ov{z}^j}.$$

\begin{exer} \label{ex15}
Show that the condition (\ref{kahlercondition}) is independent of choice of complex coordinates.
\end{exer}

\begin{example}
If $M$ is a complex manifold of complex dimension 1 (a Riemann surface) then every Hermitian metric is K\"ahler, since (\ref{kahlercondition}) is vacuous.
\end{example}

\begin{example}
$\mathbb{C}^n$ with the Euclidean metric $g_{i\ov{j}} =\delta_{ij}$ is  K\"ahler.  More generally, we may take $g_{i\ov{j}} = A_{ij}$ where $(A_{ij})$ is any fixed $n\times n$ positive definite Hermitian matrix.  Since $g_{i\ov{j}}$ is constant, it descends to the quotient $\mathbb{C}^n/\mathbb{Z}^{2n}$ to give a K\"ahler metric on the torus.
\end{example}

\begin{exer} \label{exfs}
Following on from Example \ref{examplePn}, we can define 
$$g_{i\ov{j}} = \partial_i \partial_{\ov{j}} \log (1+|z^1|^2+ \cdots + |z^n|^2), \quad \textrm{on } U_0,$$
and similarly for $U_1, \ldots, U_n$.  Show that $(g_{i\ov{j}})$ defines a K\"ahler metric on $\mathbb{P}^n$.  This metric is called the \emph{Fubini-Study metric}.
\end{exer}

Given a K\"ahler metric $g$, we define its \emph{K\"ahler form} to be
$$\omega = \sqrt{-1} g_{i\ov{j}}dz^i \wedge d\ov{z}^j.$$
Observe that since $(g_{i\ov{j}})$ is Hermitian, the form $\omega$ is real:
$$\ov{\omega} = - \sqrt{-1} \ov{g_{i\ov{j}}} d\ov{z}^i \wedge dz^j = \sqrt{-1} g_{j\ov{i}}dz^j \wedge d\ov{z}^i = \omega.$$
It is a form of type $(1,1)$ (in the span of the $dz^i \wedge d\ov{z}^j$).  

Note that if $g$ is just a Hermitian metric then one can still define an associated real $(1,1)$ form $\omega$, which is sometimes referred to as the \emph{fundamental 2-form} of $g$.  Abusing notation slightly, we will often refer to the form $\omega$ as a Hermitian metric (or as a K\"ahler metric, if $g$ is K\"ahler).

The following exercise shows that a Hermitian metric $g$ is K\"ahler if and only if $\omega$ is $d$-closed.
First we need some notation.  We define an operator $\partial$ which takes a $(p,q)$ form to a $(p+1, q)$ form as follows.
  Given a $(p,q)$ form $$a = a_{i_1\cdots i_p \ov{j_1} \cdots \ov{j_q}} dz^{i_1} \wedge \cdots \wedge dz^{i_p} \wedge d\ov{z}^{j_1} \wedge \cdots \wedge d\ov{z}^{j_{q}},$$ we define
  $$\partial a = \partial_k a_{i_1\cdots i_p \ov{j_1} \cdots \ov{j_q}} dz^k \wedge dz^{i_1} \wedge \cdots \wedge dz^{i_p} \wedge d\ov{z}^{j_1} \wedge \cdots \wedge d\ov{z}^{j_{q}}.$$
  An argument similar to the solution of Exercise \ref{ex15} shows that $\partial$ is well-defined independent of choice of coordinates.
The operator $\ov{\partial}$, taking $(p,q)$ forms to $(p,q+1)$ forms is defined similarly, and we set
$$d=\partial +\ov{\partial},$$
which is the same as the usual exterior derivative on manifolds.  

\begin{exer} \label{ex17} Let $g$ be a Hermitian metric.  Show that 
$$\textrm{$g$ is K\"ahler} \ \Longleftrightarrow \ d\omega=0 \ \Longleftrightarrow \ \partial \omega =0 \ \Longleftrightarrow \ \ov{\partial} \omega =0.$$
\end{exer}

The next result gives a way to construct new K\"ahler manifolds from old ones.

\begin{prop}
Let $(M, g)$ be a K\"ahler manifold, and $N\subset M$ a complex submanifold.  Then $g|_N$ is a K\"ahler metric on $N$.
\end{prop} 
\begin{proof}
$N$ being a complex submanifold of dimension $k$ means the following: at any point $p$ of $N$ we can find complex coordinates for $M$ centered at $p$ so that near $p$, $N$ is given by $\{ z^{k+1}= \cdots = z^n =0 \}$ and  $z^1, \ldots, z^k$ give complex coordinates for $N$.  It follows immediately that the metric $g$ at $p$ defines an inner product on the complexified tangent space of $N$ (just restrict to the span of $\partial_1, \cdots, \partial_k$).

Let $\iota: N \rightarrow M$ be the inclusion map, so that $\omega|_{N} = \iota^* \omega$.  Then by the standard property of the exterior derivative, we have
$$d\iota^*\omega = \iota^* d\omega =0,$$
as required.  
\end{proof}

It follows that:

\begin{cor}
Every smooth projective variety admits a K\"ahler metric.
\end{cor}

Indeed, a smooth projective variety can be defined to be a complex submanifold of $\mathbb{P}^N$ for some $N$, which by Example \ref{exfs} admits a K\"ahler metric.

We now define the notion of \emph{covariant differentiation} on a K\"ahler manifold $(M,g)$.  Observe that, by the argument of Exercise \ref{exh}, if we are given a $T^{1,0}$ vector field $X=X^i \partial_i$, the object
$$\frac{\partial X^i}{\partial \ov{z}^{\ell}} ,$$
gives a well-defined tensor, meaning that it transforms according to the rule
$$\frac{\partial X^i}{\partial \ov{z}^\ell}  = \frac{\partial \tilde{X}^k}{\partial \ov{\tilde{z}}^{j}}  \ov{\frac{\partial \tilde{z}^j}{\partial z^{\ell}} } \frac{\partial z^i}{\partial \tilde{z}^k} \quad \textrm{on } U \cap \tilde{U}.$$
Indeed, this follows from the fact that the transition maps for the $X^i$ are holomorphic and so ``pass through'' the operator $\partial_{\ov{\ell}}$.  However, the reader can check that the object
$$\frac{\partial X^i}{\partial z^k} $$
does  \emph{not} give a well-defined tensor, and this leads us to define covariant differentiation.

We define the \emph{Christoffel symbols} of $g$ on the chart $(U,z)$ to be the functions
$\Gamma^i_{kp}: U \rightarrow \mathbb{C}$ defined by
\begin{equation} \label{christoffel}
\Gamma^i_{kp} = g^{i\ov{q}} \partial_k g_{p\ov{q}},
\end{equation}
where we recall that $g^{i\ov{q}}$ are the components of the inverse of $g$.  By the K\"ahler condition (\ref{kahlercondition}), 
$$\Gamma^i_{kp} = \Gamma^i_{pk}.$$
The Christoffel symbols do \emph{not} define a tensor.  However, given a $T^{1,0}$ vector field $X$, we define the \emph{covariant derivative} $\nabla_k X^i$ by
$$\nabla_k X^i = \partial_p X^i + \Gamma^i_{kp} X^p,$$
and this \emph{does} define a tensor.  By the above observation, $\partial_{\ov{\ell}} X^i$ is already a tensor and we define
$$\nabla_{\ov{\ell}} X^i = \partial_{\ov{\ell}}X^i.$$
Similarly,  for a $T^{0,1}$ vector field $Y=Y^{\ov{j}} \partial_{\ov{j}}$, a $(1,0)$ form $a=a_i dz^i$ and a $(0,1)$ form $b=b_{\ov{j}}d\ov{z}^j$, we define
\[
\begin{array} {lll}
& \nabla_k Y^{\ov{j}} = \partial_k Y^{\ov{j}} \quad & \nabla_{\ov{\ell}} Y^{\ov{j}} = \partial_{\ov{\ell}}Y^{\ov{j}} + \ov{\Gamma^j_{\ell q}} Y^{\ov{q}} \\
& \nabla_k a_i = \partial_k a_i - \Gamma_{ki}^p a_p \quad & \nabla_{\ov{\ell}} a_i = \partial_{\ov{\ell}} a_i \\
& \nabla_k b_{\ov{j}} = \partial_k b_{\ov{j}} \quad & \nabla_{\ov{\ell}} b_{\ov{j}} = \partial_{\ov{\ell}} b_{\ov{j}} - \ov{\Gamma^q_{\ell j}} b_{\ov{q}}.
\end{array}
\]

\begin{exer} \label{ex18}
Show that $\nabla_k X^i$, $\nabla_{\ov{\ell}}Y^{\ov{j}}$ etc. all define tensors.
\end{exer}

Moreover, we can extend covariant differentiation naturally to any kind of tensor, such as the tensor $S^{ab}_{\ov{c}}$ described above in Section \ref{sectionvf}:
\[
\begin{split}
\nabla_k S^{ab}_{\ov{c}} = {} & \partial_k S^{ab}_{\ov{c}} + \Gamma^a_{kp} S^{pb}_{\ov{c}} + \Gamma_{kp}^b S^{ap}_{\ov{c}}\\
\nabla_{\ov{\ell}} S^{ab}_{\ov{c}} = {} & \partial_{\ov{\ell}} S^{ab}_{\ov{c}} - \ov{\Gamma^{q}_{\ell c}} S^{ab}_{\ov{q}}.
\end{split}
\]
In particular, we have $\nabla_k g_{i\ov{j}} = 0$.  Indeed, this follows from the choice of the Christoffel symbols, since
$$\nabla_k g_{i\ov{j}} = \partial_k g_{i\ov{j}} - \Gamma^p_{ki} g_{p\ov{j}} = \partial_k g_{i\ov{j}} - g^{p\ov{q}} (\partial_k g_{i\ov{q}}) g_{p\ov{j}} =0,$$
where we have used the fact that $g^{p\ov{q}} g_{p\ov{j}} =\delta_{jq}$.

\begin{rem}
 $\nabla$ coincides with the Levi-Civita connection of the Riemannian metric $g_R$ associated with $g$, extended to the complexified tangent bundle.  
\end{rem}

\section{Curvature}

We now describe the curvature associated to a K\"ahler metric $g$.  Define the \emph{curvature tensor} $R_{i \ov{j} k}^{\ \ \ \, p}$ by
$$R_{i \ov{j} k}^{\ \ \ \, p} = - \partial_{\ov{j}} \Gamma^p_{ik},$$
for $\Gamma^p_{ik}$ the Christoffel symbols of $g$, defined by (\ref{christoffel}).  

\begin{exer} \label{ex119}
Show that $R_{i\ov{j}k}^{\ \ \ \, p}$ is a tensor.
\end{exer}

It will be convenient to define
$$R_{i\ov{j} k \ov{\ell}} = R_{i\ov{j} k}^{\ \ \ \, p} g_{p\ov{\ell}}.$$
That is, we lower the index $p$ into the last slot using the metric $g$.   We will also refer to this tensor as the curvature tensor. The tensor $R_{i\ov{j} k \ov{\ell}}$ has the following symmetries:

\begin{prop} \label{propsym}  The curvature tensor of a K\"ahler metric satisfies
$$R_{i\ov{j}k\ov{\ell}} = R_{k\ov{j} i \ov{\ell}} = R_{i \ov{\ell} k \ov{j}} = R_{k\ov{\ell} i \ov{j}},$$
and
$$\ov{R_{i\ov{j}k\ov{\ell}}} = R_{j \ov{i} \ell \ov{k}}.$$
\end{prop}
\begin{proof}
From the definitions,
\begin{equation} \label{Rf}
\begin{split}
R_{i\ov{j}k\ov{\ell}} = {} & - g_{p\ov{\ell}} \partial_{\ov{j}} (g^{p\ov{q}} \partial_i g_{k\ov{q}}) \\
={} & - g_{p\ov{\ell}} g^{p\ov{q}} \partial_{\ov{j}} \partial_i  g_{k\ov{q}} + g_{p\ov{\ell}} g^{p\ov{s}} g^{r\ov{q}} \partial_{\ov{j}}g_{r\ov{s}} \partial_i g_{k\ov{q}} \\
={} & - \partial_i \partial_{\ov{j}} g_{k\ov{\ell}} + g^{p\ov{q}} \partial_i g_{k\ov{q}} \partial_{\ov{j}} g_{p\ov{\ell}},
\end{split}
\end{equation}
where we have used the formula for the derivative of an inverse matrix $\delta (A^{-1}) = -A^{-1} (\delta A) A^{-1}.$
The proposition then follows immediately from this formula, and the K\"ahler condition (\ref{kahlercondition}).
\end{proof}

The curvature tensor measures the failure of covariant derivatives to commute.  More precisely:

\begin{prop} \label{propcom}
For a $T^{1,0}$ vector field $X=X^p \partial_p$, a $T^{0,1}$ vector field $Y=Y^{\ov{q}} \partial_{\ov{q}}$, a $(1,0)$ form $a=a_p dz^p$ and a $(0,1)$ form $b=b_{\ov{q}}d\ov{z}^q$, we have the following commutation formulae:
\[
\begin{array}{lll}
& [\nabla_i, \nabla_{\ov{j}}] X^p =  R_{i\ov{j}k}^{\ \ \ \, p} X^k,  \qquad & [\nabla_i, \nabla_{\ov{j}} ] Y^{\ov{q}} =  - R_{i\ov{j}\ \, \, \ov{\ell}}^{\ \ \, \ov{q}} Y^{\ov{\ell}} \\
 & [\nabla_i, \nabla_{\ov{j}} ] a_p = - R_{i\ov{j}p}^{\ \ \ \,  q} a_q, \qquad & [\nabla_i, \nabla_{\ov{j}}] b_{\ov{q}} = R_{i\ov{j}\ \, \ov{q}}^{\ \ \, \ov{\ell}} \, b_{\ov{\ell}}.
\end{array}
\]
Here, $[\nabla_i, \nabla_{\ov{j}}] = \nabla_i \nabla_{\ov{j}} - \nabla_{\ov{j}} \nabla_i$, and we are raising and lowering indices of the curvature tensor using $g$.
\end{prop}

Before we prove this proposition, it is convenient to introduce the notion of a \emph{holomorphic normal coordinate system}.

\begin{lem} \label{lemnormal}
Let $(M, g)$ be a K\"ahler manifold.  For any fixed point $x \in M$, there exists a holomorphic coordinate chart $(U, z)$ centered at $x$ such that, at $x$,
$$g_{i\ov{j}} = \delta_{ij}, \quad \textrm{and} \quad \partial_k g_{i\ov{j}} =0,$$
for all $i,j,k=1, 2, \ldots, n.$
\end{lem}
\begin{proof}
By an affine linear change in coordinates, we can find coordinates $\tilde{z}^1, \ldots, \tilde{z}^n$ centered at $x$ with $\tilde{g}_{i\ov{j}}=\delta_{ij}$ at that point.  To obtain the vanishing of the first derivatives of $g$, define a new holomorphic coordinate system $z^1, \ldots, z^n$ by
$$\tilde{z}^i = z^i - \frac{1}{2}  \tilde{\Gamma}_{jk}^i(0) z^j z^k.$$
Observe that the first derivative of $\tilde{z}=\tilde{z}(z)$ at $0$ is the identity, and hence by the inverse function theorem, we can solve for $z$ as a holomorphic function of $\tilde{z}$ in a neighborhood of zero.  We leave it as an exercise  to check that in the $z$ coordinate system, we have
$$\quad \partial_k g_{i\ov{j}} =0$$
at $x$ for all $i,j,k$.
\end{proof}

\begin{exer} \label{ex110}
Complete the proof of Lemma \ref{lemnormal}
\end{exer}

The coordinate system we constructed in Lemma \ref{lemnormal} is called a \emph{holomorphic normal coordinate system} for $g$. The lemma implies in particular that we can choose coordinates for which the Christoffel symbols vanish at a point (and this implies that $\Gamma$ cannot be a tensor, since if it were, it would have to vanish everywhere).
 Note that the K\"ahler condition is required for the existence of holomorphic normal coordinates.  Indeed from (\ref{kahlercondition}) it is immediate that the existence of these coordinates for $g$ implies that $g$ is K\"ahler.  

We now complete the proof of Proposition \ref{propcom}.

\begin{proof}[Proof of Proposition \ref{propcom}]
Since both sides are tensors, it is sufficient to prove the identities at a single point $x$, in a holomorphic normal coordinate system (if an equation of tensors holds in one coordinate system, it must hold in every coordinate system). Compute at $x$,
\[
\begin{split}
[\nabla_i, \nabla_{\ov{j}}] X^p = {} & \partial_i \nabla_{\ov{j}} X^p - \partial_{\ov{j}} ( \partial_i X^p + \Gamma^p_{ik} X^k) \\
={} & \partial_i \partial_{\ov{j}} X^p - \partial_{\ov{j}} \partial_i X^p - (\partial_{\ov{j}} \Gamma^p_{ik} )X^k \\
={} & R_{i\ov{j} k}^{ \ \ \  \, p} X^k, 
\end{split}
\]
giving the first formula.  The others are left as the next exercise.
\end{proof}

\begin{exer} \label{ex111}
Complete the proof of Proposition \ref{propcom}.
\end{exer}

\begin{exer} \label{extr}
Let $\omega = \sqrt{-1} g_{i\ov{j}} dz^i \wedge d\ov{z}^j$ be a K\"ahler metric.
\begin{enumerate}
\item Show that if $\beta = \sqrt{-1} \beta_{i\ov{j}} dz^i \wedge d\ov{z}^j$ is a real $(1,1)$-form and $\omega$ a K\"ahler form, then
$$n \omega^{n-1} \wedge \beta = g^{i\ov{j}} \beta_{i\ov{j}} \omega^n =: (\tr{\omega}{\beta}) \omega^n.$$
\item Let $f$ be a real-valued function.  Show that
$$n \omega^{n-1} \wedge \sqrt{-1} \partial f \wedge \ov{\partial} f = | \partial f|^2_g \omega^n.$$
\end{enumerate}
\emph{Hint:  pick coordinates at a point for which $g_{i\ov{j}} =\delta_{ij}$.}
\end{exer}

We end this section by discussing the 
 \emph{Ricci curvature} of a K\"ahler metric, which is defined to be the tensor $R_{i\ov{j}}$ given by 
$$R_{i\ov{j}} = g^{k\ov{\ell}} R_{i\ov{j} k\ov{\ell}}.$$

A key property of K\"ahler metrics is the following simple formula for the Ricci curvature:

\begin{prop} \label{propRic} The Ricci curvature is given by
$$R_{i\ov{j}} = -\partial_i \partial_{\ov{j}} \log \det g.$$
\end{prop}
\begin{proof}
The proposition follows easily from the well-known formula for the derivative of the determinant of an invertible Hermitian matrix $A$:
$$\delta \det A = \textrm{trace}(A^{-1}\delta A) \det A,$$
which can be rewritten as
\begin{equation}\label{logdet}
\delta \log \det A = \textrm{trace}(A^{-1}\delta A).
\end{equation}
We compute
$$R_{i\ov{j}} = - g^{k\ov{\ell}} g_{p\ov{\ell}} \partial_{\ov{j}} \Gamma^p_{ik} = - \partial_{\ov{j}} \Gamma^p_{ip} = - \partial_{\ov{j}} (g^{p\ov{q}} \partial_i g_{p\ov{q}}) = - \partial_{\ov{j}} \partial_i \log \det g,$$
as required.
\end{proof}

We define the \emph{Ricci form} of $g$ to be the $(1,1)$ form
$$\textrm{Ric}(\omega) = \sqrt{-1} R_{i\ov{j}}dz^i \wedge d\ov{z}^j = - \ddbar \log \det g.$$
It follows from either Proposition \ref{propsym} or Proposition \ref{propRic} that $(R_{i\ov{j}})$ is Hermitian, and hence $\textrm{Ric}(\omega)$ is a real $(1,1)$ form.  Proposition \ref{propRic}
implies that  $\partial_k R_{i\ov{j}} = \partial_i R_{k\ov{j}}$, namely that $\textrm{Ric}(\omega)$ is $d$-closed.

Note that we often write
$$\textrm{Ric}(\omega) = - \ddbar \log \omega^n.$$
We make sense of this expression as follows.  If $\Omega$ is any volume form, locally written as
$$\Omega = a(z) (\sqrt{-1})^n dz^1 \wedge d\ov{z}^1 \wedge \cdots \wedge dz^n \wedge d\ov{z}^n,$$
then we define
$$\ddbar \log \Omega = \ddbar \log a.$$

\begin{exer} \label{exOmega}
This definition of $\ddbar \log \Omega$ is well-defined, independent of choice of local coordinates.
\end{exer}

Then since 
$$\omega^n = n! (\sqrt{-1})^n \det g \, dz^1 \wedge d\ov{z}^1 \wedge \cdots \wedge dz^n \wedge d\ov{z}^n,$$
we see that $- \ddbar \log \omega^n = -\ddbar \log \det g$.

\begin{exer} \label{exfs2}
Let $\omega_{\textrm{FS}}$ be the Fubini-Study metric of Exercise \ref{exfs}.  Show that
$$\textrm{Ric}(\omega_{\textrm{FS}}) = (n+1) \omega_{\textrm{FS}}.$$
\end{exer}

\chapter{The K\"ahler-Ricci flow and the K\"ahler cone}

In this lecture we introduce the K\"ahler-Ricci flow.  We also discuss K\"ahler classes,  the K\"ahler cone and the first Chern class.  We describe the maximal existence time result for the K\"ahler-Ricci flow  and give some simple examples.

\section{The K\"ahler-Ricci flow and simple examples}

Let $(M, \omega_0)$ be a compact K\"ahler manifold.  If $\omega=\omega(t)$ is a smooth family of K\"ahler metrics on $M$ satisfying the equation
\begin{equation} \label{krf}
\ddt{} \omega = - \textrm{Ric}(\omega), \quad \omega|_{t=0} = \omega_0,
\end{equation}
then we say that $\omega(t)$ is a solution of the \emph{K\"ahler-Ricci flow} starting at $\omega_0$.  For the reader who is familiar with Hamilton's Ricci flow of Riemannian metrics:  this is the same equation (modulo a factor of $2$) starting at a K\"ahler metric.

We describe now some simple examples of solutions to the K\"ahler-Ricci flow.  First, let $M$ be a compact Riemann surface (complex dimension 1).  The following theorem is known as the \emph{Uniformization Theorem}.

\begin{thm} \label{uniformization}
On any compact Riemann surface $M$ there exists a K\"ahler metric $\omega_{\emph{KE}}$ with
\begin{equation} \label{KE}
\emph{Ric}(\omega_{\emph{KE}}) = \mu \, \omega_{\emph{KE}},
\end{equation}
for some constant $\mu$.
\end{thm}

In general, a K\"ahler metric satisfying (\ref{KE}) is called a \emph{K\"ahler-Einstein} metric, which explains the notation.  By multiplying $\omega_{\emph{KE}}$ by a constant, we may assume that $\mu$ is equal to either  $1$, $0$ or $-1$.  Indeed this follows from the fact that
for any K\"ahler metric $\omega$ and positive real number $\lambda$,
\begin{equation} \label{scaleRic}
\textrm{Ric}(\lambda \omega) = \textrm{Ric}(\omega),
\end{equation} 
as can be seen immediately from the formula of Proposition \ref{propRic}.  

The Gauss-Bonnet formula on a Riemann surface $M$ can be written as
$$\int_M \textrm{Ric}(\omega) = 2\pi (2-2g_M),$$
where $g_M$ is the genus of $M$, and hence the sign of $\mu$ determines whether $M$ has genus 0, 1 or greater than 1.

\begin{example} \label{P1krf}
If $\mu=1$ then $M = \mathbb{P}^1$.  Let $\omega_0 = 2 \omega_{\textrm{FS}}$ where $\omega_{\textrm{FS}}$ is the Fubini-Study metric from Exercise \ref{exfs}.  Then by Exercise \ref{exfs2},
$$\textrm{Ric}(\omega_0) = \omega_0.$$
We claim that $\omega(t) = (1-t) \omega_0$ is a solution of the K\"ahler-Ricci flow on $[0,1)$.  Indeed, 
$$\ddt{} \omega(t) = - \omega_0 = - \textrm{Ric}(\omega_0) = - \textrm{Ric}(\omega(t)),$$
where the last equality makes use of (\ref{scaleRic}).
Hence  there is a solution of the K\"ahler-Ricci flow which shrinks the Fubini-Study metric to zero in finite time, by scaling.  
Recall that $\mathbb{P}^1$ is diffeomorphic to $S^2$  (Exercise \ref{S2}).  In fact, the Fubini-Study metric is a constant multiple of the standard round metric on $S^2$, and so this solution of the K\"ahler-Ricci flow can be visualized as a shrinking round sphere (Figure \ref{fp1}).
\end{example}

\begin{figure}[h!]
\begin{center}
\begin{tikzpicture}
\shade [ball color=black!10] (0,0) circle [radius=1cm];
\draw[->, thin] (1.3,0) -- (2.0,0);
\shade [ball color=black!10] (3,0) circle [radius=.65cm];
\draw[->, thin] (4.0,0) -- (4.8,0);
\shade [ball color=black!10] (5.5,0) circle [radius=.3cm];
\draw[->, thin] (6.2,0) -- (6.7,0);
\draw[fill] (7,0) circle (.3pt);
\end{tikzpicture}  
\end{center}
\caption{$\mathbb{P}^1$ shrinking along the K\"ahler-Ricci flow} \label{fp1}
\end{figure}
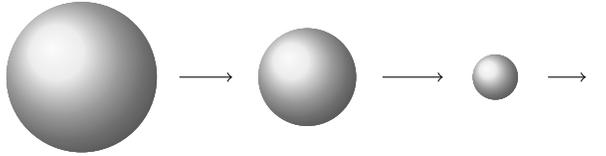

\begin{example} \label{T2krf}
If $\mu=0$ then $M$ is a torus.  We have a stationary solution of the K\"ahler-Ricci flow starting at $\omega_0 = \omega_{\textrm{KE}}$,
$$\omega(t) = \omega_0,$$
for $t \ge 0$.
\end{example}

\begin{example} \label{gg1krf}
If $\mu=-1$ then $M$ is  a surface of genus strictly greater than one.  If  $\omega_{\textrm{KE}}$ satisfies $\textrm{Ric}(\omega_{\textrm{KE}}) = - \omega_{\textrm{KE}},$ then
$$\omega(t) = (1+t)\omega_{\textrm{KE}}$$
solves the K\"ahler-Ricci flow for $t \ge 0$ starting at $\omega_0 = \omega_{\textrm{KE}}$.  The solution of the K\"ahler-Ricci flow exists for all time and expands by scaling. 
\end{example}

\section{The K\"ahler cone and the first Chern class}

A K\"ahler metric $\omega$ is a closed real $(1,1)$ form, and hence defines an element of the cohomology group
$$H^{1,1}_{\ov{\partial}}(M, \mathbb{R}) = \frac{ \{ \textrm{$\ov{\partial}$-closed real $(1,1)$ forms} \}}{\textrm{Im} \, \ov{\partial}}.$$
By Hodge theory, $H^{1,1}_{\ov{\partial}}(M, \mathbb{R})$ is a finite dimensional vector space over $\mathbb{R}$.

The \emph{$\partial \ov{\partial}$ Lemma}, which holds on K\"ahler manifolds, and which we will not state in its full generality, implies that
$$H^{1,1}_{\ov{\partial}}(M, \mathbb{R})  = \frac{ \{ \textrm{$\ov{\partial}$-closed real $(1,1)$ forms} \}}{\textrm{Im} \, \partial \ov{\partial}}.$$
Namely, if $\beta$ and $\gamma$ are two closed real $(1,1)$ forms with $\beta = \gamma + \ov{\partial} \eta$ for some $(1,0)$ form $\eta$ then
$\beta = \gamma + \ddbar f$ for a real-valued function $f$.

In particular, if $\omega$ and $\omega'$ are two K\"ahler metrics with $[\omega]=[\omega']$ (i.e. they define the same element in $H^{1,1}_{\ov{\partial}}(M, \mathbb{R})$) then
\begin{equation} \label{oop}
\omega' = \omega+ \ddbar \varphi,
\end{equation}
for some smooth real-valued function $\varphi$.  Moreover, the function $\varphi$ is unique up to a constant, since if $\tilde{\varphi}$ is another function satisfying (\ref{oop}) then $\ddbar (\varphi - \tilde{\varphi})=0$ and by the next exercise, $\varphi - \tilde{\varphi}$ is a constant.

\begin{exer}  \label{exerddbar} Let $(M, \omega)$ be a compact K\"ahler manifold.
Show that if a smooth function $f: M \rightarrow \mathbb{R}$ satisfies $\ddbar f \ge 0$ then $f$ is a constant on $M$.  \emph{Hint:  integrate $f \ddbar f \wedge \omega^{n-1}$ over $M$ and use Stokes' Theorem.}
\end{exer}

We say that a class $\alpha$ in $H^{1,1}_{\ov{\partial}}(M, \mathbb{R})$ is a \emph{K\"ahler class} if there exists a K\"ahler metric $\omega$ with $[\omega]= \alpha$, and in this case we write $\alpha>0$.  If $-\alpha$ is a K\"ahler class then we write $\alpha<0$.  
\begin{exer} \label{ex22}
For $\alpha$ in $H^{1,1}_{\ov{\partial}}(M, \mathbb{R})$, show that the conditions $\alpha>0$, $\alpha=0$ and $\alpha<0$ are mutually exclusive.  Here $\alpha=0$ simply means that $\alpha$ is the zero element of $H^{1,1}_{\ov{\partial}}(M, \mathbb{R})$.
\end{exer}

Note that a class $\alpha$ in $H^{1,1}_{\ov{\partial}}(M, \mathbb{R})$ need not satisfy one of $\alpha>0$, $\alpha=0$ or $\alpha<0$, as we shall see in examples later.

We define the \emph{K\"ahler cone} of $M$ to be
$$\textrm{Ka}(M) = \{ \alpha \in H^{1,1}_{\ov{\partial}}(M, \mathbb{R}) \ | \ \alpha>0 \}.$$

\begin{exer} \label{ex23}
Show that $\textrm{Ka}(M)$ is an open convex cone in $H^{1,1}_{\ov{\partial}}(M, \mathbb{R})$.  (Recall that being a \emph{convex cone} means that $\alpha, \alpha' \in \textrm{Ka}(M)$ and $s, s' \in \mathbb{R}^{>0}$ implies that $s\alpha + s'\alpha'  \in \textrm{Ka}(M)$.)
\end{exer}

We now describe the \emph{first Chern class} of a K\"ahler manifold $M$.  This is a special element of $H^{1,1}_{\ov{\partial}}(M, \mathbb{R})$ defined by
$$c_1(M) = [\textrm{Ric}(\omega)],$$
where $\omega$ is any K\"ahler metric on $M$.  Note that, in comparison to the usual definition in the literature,  we have  omitted a factor of $2\pi$.

It appears from the definition that $c_1(M)$ depends on the choice of metric $\omega$, but in fact it does not:

\begin{prop}
$c_1(M)$ is independent of choice of $\omega$.
\end{prop}
\begin{proof}
Let $\omega'= \sqrt{-1} g'_{i\ov{j}}dz^i \wedge d\ov{z}^j$ be any other K\"ahler metric.  Then
$$\det g' = e^F \det g,$$
for some smooth function $F: M \rightarrow \mathbb{R}$.  Then
\[
\textrm{Ric}(\omega') = - \ddbar \log \det g' = - \ddbar \log \left( e^F \det g \right) = \textrm{Ric}(\omega) - \ddbar F,
\]
which implies that $[\textrm{Ric}(\omega')] = [\textrm{Ric}(\omega)]$.
\end{proof}

Note that the manifolds in Examples \ref{P1krf}, \ref{T2krf} and \ref{gg1krf} have $c_1(M)>0$, $c_1(M)=0$ and $c_1(M)<0$ respectively.

\begin{exer} \label{prod}
 Let $M = M_1 \times M_2$ be a product of two K\"ahler manifolds $(M_1, \omega_1)$ and $(M_2, \omega_2)$, and write $\pi_1:M \rightarrow M_1$ and $\pi_2: M \rightarrow M_2$ for the projection maps.  Let $\omega = \pi_1^* \omega_1 + \pi_2^* \omega_2$ be the product of the two metrics $\omega_1$ and $\omega_2$,  a K\"ahler metric on $M$.  
 \begin{enumerate}
\item[(a)]  Show that
$$\textrm{Ric}(\omega) = \pi_1^* \textrm{Ric}(\omega_1) + \pi_2^* \textrm{Ric}(\omega_2).$$
\item[(b)]  Suppose that $\omega_1(t)$ and $\omega_2(t)$ solve the K\"ahler-Ricci flow on $M_1$ and $M_2$ respectively, starting at $\omega_1$ and $\omega_2$.  Then show that $\omega(t) = \pi_1^* \omega_1(t) + \pi_2^* \omega_2(t)$ solves the K\"ahler-Ricci flow on $M$.
\end{enumerate}
\end{exer}

\begin{exer} \label{exerP1P1}
Let $M = \mathbb{P}^1 \times \mathbb{P}^1$ and write $\omega_{\textrm{KE}} = 2 \omega_{\textrm{FS}}$ for the K\"ahler-Einstein metric on $\mathbb{P}^1$, which we recall has $\textrm{Ric}(\omega_{\textrm{KE}})=\omega_{\textrm{KE}}$.  Let $\omega$ be the K\"ahler metric on $M$ given by the product of these K\"ahler-Einstein metrics on $\mathbb{P}^1$.  Show that
$$c_1(M) = [\omega].$$
\end{exer}

\section{Maximal existence time for the K\"ahler-Ricci flow} \label{sectionmax1}

We now return to the K\"ahler-Ricci flow (\ref{krf}), and observe that if $\omega(t)$ is a solution of the  flow then the cohomology classes $[\omega]=[\omega(t)]$ must evolve by
\begin{equation} \label{krfc}
\frac{d}{dt} [\omega] = - c_1(M), \quad [\omega]|_{t=0} = [\omega_0].
\end{equation}
This simple ODE system has the solution
$$[\omega(t)] = [\omega_0] - tc_1(M).$$
Hence, as long as a solution to the K\"ahler-Ricci flow exists, we must have
$$[\omega_0] - t c_1(M) >0,$$
since this element of $H^{1,1}_{\ov{\partial}}(M, \mathbb{R})$ contains the K\"ahler metric $\omega(t)$.  The maximal existence time theorem for the K\"ahler-Ricci flow states that this necessary condition is sufficient for existence of a solution:

\begin{thm} \label{thmmaximal}
There exists a unique maximal solution to the K\"ahler-Ricci flow (\ref{krf}) starting at $\omega_0$ for $t \in [0,T)$, where
\begin{equation} \label{deT}
T = \sup \{ t >0 \ | \ [\omega_0] - tc_1(M)>0\}.
\end{equation}
\end{thm}

We say that a solution $\omega(t)$ for $t\in [0,T)$ to the K\"ahler-Ricci flow starting at $\omega_0$  is \emph{maximal} if there does not exist a solution starting at $\omega_0$ on $[0,T')$ for any $T'>T$.

The result Theorem \ref{thmmaximal} is due to Cao \cite{Cao} in the special case when $c_1(M)$ is zero, positive or negative.  In this generality, it was proved by Tian-Zhang \cite{TZha}; weaker versions of the result appeared earlier in the work of Tsuji \cite{Ts1, Ts2}.

Theorem \ref{thmmaximal} says that the flow exists for as long as the straight line path $t \mapsto [\omega_0] - tc_1(M)$ remains in the K\"ahler cone.  There are four possibilities:
\begin{enumerate}
\item[(a)] The path  $t \mapsto [\omega_0] - tc_1(M)$ hits zero.  This can only occur if $c_1(M)>0$ and $[\omega_0] = Tc_1(M)$ with $T<\infty$. 
\item[(b)] The K\"ahler class does not move.  This occurs if and only if $c_1(M)=0$.
\item[(c)]  The path $t \mapsto [\omega_0] - tc_1(M)$ remains in the K\"ahler cone for all time.  This could occur if $c_1(M)<0$, for example.
\item[(d)]  The path $t \mapsto [\omega_0] - tc_1(M)$ hits a non-zero element of the boundary of the K\"ahler cone.  This kind of behavior often occurs, as we will discuss later.  The behavior of the flow will depend on the kind of boundary element that the path hits (see Example \ref{P1P1e} for a simple illustration of this.)
\end{enumerate}
These are illustrated by Figure \ref{fkrf}.

\medskip

\begin{figure}[h!] 

\begin{tikzpicture}[scale=5, important line/.style={thick},
    dashed line/.style={dashed,->, thick},
        every node/.style={color=black} ]
\draw[important line] (0,0) coordinate (A) -- (.4,.8) coordinate (B) node[right, text width=5em] {};
\draw[important line] (-.4,.8) coordinate (B) node[left, text width=2em]{(a)} -- (0,0) coordinate (A) node[right, text width=5em]{$0$};
\draw[fill] (.05,.45) circle (.3pt);
\draw[directed, dashed, thick] (.05, .45) coordinate (C) node[right, text width=5em] {$[\omega_0]$} -- (0,0) coordinate (A);

\draw[important line] (1.25,0) coordinate (D) -- (1.65,.8) coordinate (E) node[right, text width=5em] {};
\draw[important line] (.85,.8) coordinate (F) node[left, text width=2em]{(b)} -- (1.25,0) coordinate (D) node[right, text width=5em]{$0$};
\draw[fill] (1.3,.45) circle (.3pt);
\draw[directed, dashed, thick] (1.3, .45) coordinate (G) node[right, text width=5em] {$[\omega_0]$} -- (1.3,0.45) coordinate (D);
\end{tikzpicture}    

\begin{tikzpicture}[scale=5, important line/.style={thick},
    dashed line/.style={dashed,->, thick},
        every node/.style={color=black} ]
\draw[important line] (0,0) coordinate (A) -- (.4,.8) coordinate (B) node[right, text width=5em] {};
\draw[important line] (-.4,.8) coordinate (B) node[left, text width=2em]{(c)} -- (0,0) coordinate (A) node[right, text width=5em]{$0$};
\draw[fill] (.05,.45) circle (.3pt);
\draw[directed, dashed, thick] (.05, .45) coordinate (C) node[right, text width=5em] {$[\omega_0]$} -- (-.05,.8) coordinate (A);

\draw[important line] (1.25,0) coordinate (D) -- (1.65,.8) coordinate (E) node[right, text width=5em] {};
\draw[important line] (.85,.8) coordinate (F) node[left, text width=2em]{(d)} -- (1.25,0) coordinate (D) node[right, text width=5em]{$0$};
\draw[fill] (1.3,.45) circle (.3pt);
\draw[directed, dashed, thick] (1.3, .45) coordinate (G) node[right, text width=5em] {$[\omega_0]$} -- (1,0.5) coordinate (D);
\end{tikzpicture}    

\caption{The K\"ahler-Ricci flow at the level of cohomology classes} \label{fkrf}
\end{figure}
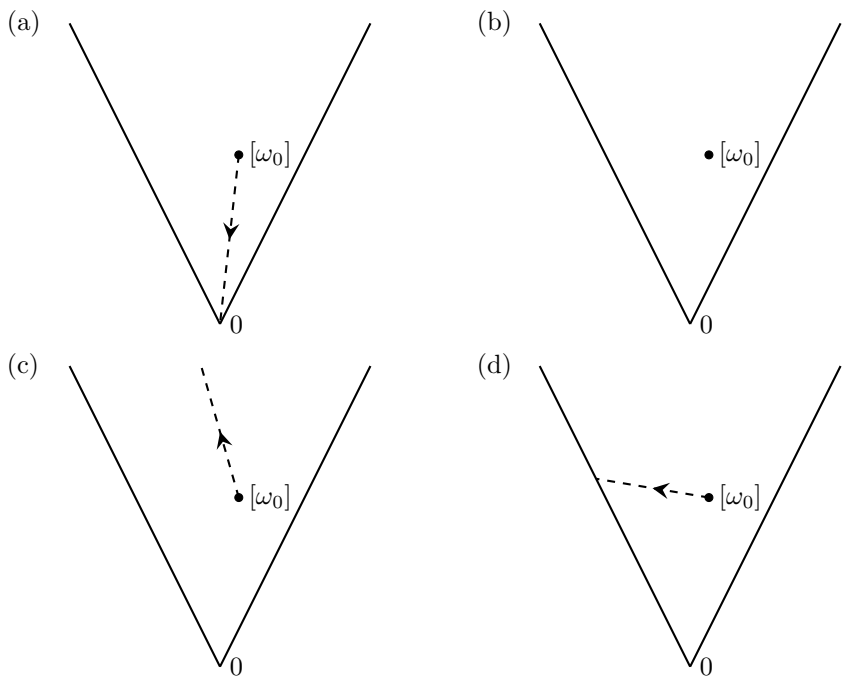

\begin{example}
Let $M$ be a Riemann surface.  Then  $H^{1,1}_{\ov{\partial}}(M, \mathbb{R})$ is one-dimensional and the K\"ahler cone is the open half line.
The three behaviors (a), (b) and (c) occur when $\mu=1$, $\mu=0$ and $\mu = -1$ respectively.
\end{example}

\begin{example} \label{P1P1e}  Let $M= \mathbb{P}^1 \times \mathbb{P}^1$ as in Exercise \ref{exerP1P1}.  The space $H^{1,1}_{\ov{\partial}}(M, \mathbb{R})$ is spanned by the classes $\alpha_i = [\pi_i^* \omega_{\textrm{KE}}]$ for $i=1,2$, using the obvious notation.  The K\"ahler cone is given by
$$\textrm{Ka}(M) = \{ x \alpha_1 + y \alpha_2 \ | \ x,y  \in \mathbb{R}^{>0} \}.$$
From Exercise \ref{exerP1P1}, the first Chern class of $M$ is $c_1(M) = \alpha_1+\alpha_2$.  Suppose that the initial metric is a product of K\"ahler-Einstein metrics 
$$\omega_0 = x \, \pi_1^* \oke + y \, \pi_2^*\oke \in x\alpha_1 + y \alpha_2.$$   From Exercise \ref{prod} we see that there are three possible behaviors of the K\"ahler-Ricci flow depending on the values of $x$, $y$.  

\medskip

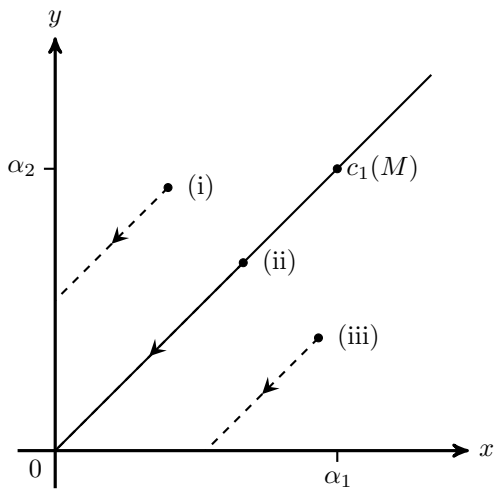
\begin{figure}[h!]
\begin{center}
\begin{tikzpicture}[scale=5, important line/.style={thick}, ax
    dashed line/.style={dashed,->, thick}, axis/.style={very thick, ->, >=stealth'},
        every node/.style={color=black} ]
        \draw[axis] (-0.1,0)  -- (1.1,0) node(xline)[right]
        {$x$};
    \draw[axis] (0,-0.1) -- (0,1.1) node(yline)[above] {$y$};
    \draw[important line] (0,0) -- (1,1);
    \draw[directed, dashed, thick] (.3, .7) node[right, text width=2em] {\ (i)}  -- (0,.4);
    \draw[directed, dashed, thick] (.5, .5) node[right, text width=2em] {\ (ii)} -- (0,.0);
    \draw[directed, dashed, thick] (.7, .3)  node[right, text width=2em] {\ (iii)} -- (0.4,.0);
    \draw[fill] (.3,.7) circle (.3pt);
    \draw[fill] (.7,.3) circle (.3pt);
    \draw[fill] (.5,.5) circle (.3pt);
    \draw[fill] (.75,.75) circle (.3pt) node[right, text width=2em] {$c_1(M)$};
    \draw (0,0) node[left, below, text width=2em] {$0$};
    \draw[important line] (0.75,0) -- (0.75,-.03) node[below, text width=1em] {$\alpha_1$};
    \draw[important line] (0,0.75) -- (-0.03,.75) node[left, text width=1em] {$\alpha_2$};
 \end{tikzpicture}    
\end{center}
\caption{Three behaviors of the K\"ahler-Ricci flow on $\mathbb{P}^1 \times \mathbb{P}^1$} \label{fp1p1}
\end{figure} 

\begin{enumerate}
\item[(i)] $[\omega_0]$ lies above the diagonal.  Then the path $t \mapsto [\omega_0] - tc_1(M)$ hits a boundary element which is a multiple of $[\alpha_2]$.  The first $\mathbb{P}^1$ shrinks to zero as $t \rightarrow T$ and the flow converges to a multiple of the K\"ahler-Einstein metric on the second $\mathbb{P}^1$.
\item[(ii)]  $[\omega_0]$ lies on the diagonal $x=y$.  Then the path $t \mapsto [\omega_0] - tc_1(M)$ hits zero and the two $\mathbb{P}^1$'s shrink simultaneously to a point in finite time.
\item[(iii)] $[\omega_0]$ lies below the diagonal.  The same behavior as in (i) with the roles of the two $\mathbb{P}^1$'s  reversed.
\end{enumerate}
These are illustrated by Figure \ref{fp1p1}.
\end{example}

\begin{exer} \label{ex26} Let $E$ be a torus and $S$ a surface of genus $>1$, with K\"ahler metrics $\omega_E$ and $\omega_S$ respectively, satisfying
$$\Ric(\omega_E)=0, \quad \Ric(\omega_S)= -\omega_S.$$
Let $M = E \times S$.
\begin{enumerate}
\item[(a)] Find $c_1(M)$ and show that  $T=\infty$ for every choice of initial metric $\omega_0$.
\item[(b)] Write down the solution $\omega(t)$ of the K\"ahler-Ricci flow starting with $\omega_0$ equal to the product of $\omega_E$ and $\omega_S$.
\item[(c)] For your solution from (b), describe geometrically what is happening to $\omega(t)/t$ as $t \rightarrow \infty$.
\end{enumerate}
\end{exer}

\begin{exer} \label{ex27} Fix a K\"ahler manifold $(M, \omega_0)$.  We say that a class $\alpha \in H_{\ov{\partial}}^{1,1}(M, \mathbb{R})$ is \emph{nef} is for all $\ve>0$ there exists $\omega_{\ve} \in \alpha$ with $\omega_{\ve} \ge - \ve \omega_0$.  
\begin{enumerate}
\item[(a)] Show that a class $\alpha$ is nef if and only if it is in the closure of the K\"ahler cone of $M$.
\item[(b)] Show that $T$ given by (\ref{deT}) can be written as
$$T = \sup\{ t>0 \ | \ [\omega_0] - tc_1(M) \textrm{ is nef} \}.$$
\end{enumerate}
\end{exer}

\chapter{The parabolic complex Monge-Amp\`ere equation}

In this lecture we describe how the K\"ahler-Ricci flow can be reduced to a parabolic complex Monge-Amp\`ere equation.  We use this description to prove the maximal existence time result for the flow.

\section{Reduction to the complex Monge-Amp\`ere equation}

We wish to prove Theorem \ref{thmmaximal} which states that there is a unique maximal solution to the K\"ahler-Ricci flow (\ref{krf})
\begin{equation*}
\ddt{} \omega = - \textrm{Ric}(\omega), \quad \omega|_{t=0} = \omega_0,
\end{equation*}
 starting at $\omega_0$ on $[0,T)$ for 
$$T = \sup \{ t >0 \ | \ [\omega_0] - tc_1(M)>0\}.$$

The idea is to reduce the K\"ahler-Ricci flow equation to a \emph{parabolic complex Monge-Amp\`ere equation}.  We will briefly discuss now this terminology.

On $\mathbb{C}^n$ the complex Monge-Amp\`ere operator is the determinant of the complex Hessian:
$$\varphi \mapsto \det \left( \frac{\partial^2 \varphi}{\partial z^i \partial \ov{z}^j} \right), \textrm{ for $\varphi$ with} \left( \frac{\partial^2 \varphi}{\partial z^i \partial \ov{z}^j} \right) \ge 0.$$
However, on a compact K\"ahler manifold $(M,g)$, this last condition is too strong, since it would imply that $\varphi$ is constant (see Exercise \ref{exerddbar}).  It is natural to replace the complex Hessian $(\partial_i \partial_{\ov{j}} \varphi)$ by $(g_{i\ov{j}} + \partial_i \partial_{\ov{j}} \varphi)$.  We define
the complex Monge-Amp\`ere operator on $M$ to be
$$\varphi \mapsto \det \left( g_{i\ov{j}} + \frac{\partial^2 \varphi}{\partial z^i \partial \ov{z}^j} \right), \textrm{ for $\varphi$ with }
\left( g_{i\ov{j}} + \frac{\partial^2 \varphi}{\partial z^i \partial \ov{z}^j} \right)\ge 0.$$
There are many functions $\varphi$ satisfying $g_{i\ov{j}} + \partial_i \partial_{\ov{j}} \varphi>0$, and indeed these functions parametrize (modulo constants) the space of K\"ahler metrics in the same cohomology class as $g$.

By \emph{parabolic complex Monge-Amp\`ere equation}, we mean a nonlinear equation of the form
$$\ddt{} \varphi = \log  \frac{\det \left( g_{i\ov{j}} + \frac{\partial^2 \varphi}{\partial z^i \partial \ov{z}^j} \right)}{\det g}+f(\varphi,t),$$
for some function $f$.   The metric $g$ may in general have some dependence on $t$.

Before we show that the K\"ahler-Ricci flow (\ref{krf}) is equivalent to such an equation,  we need the following well known  theorem of Hamilton \cite{H1} (see \cite{DeT} for a shorter proof, and \cite{CLN} for a recent exposition). 

\begin{thm} \label{hamilton} Given any compact K\"ahler manifold $(M, g_0)$,
there exists a unique solution of the K\"ahler-Ricci flow on a maximal time interval $[0, S)$ for some $S$ with $0<S \le \infty$.
\end{thm}

 In fact, Hamilton's theorem holds for the general Ricci flow on a compact  Riemannian manifold. 
Theorem \ref{hamilton} is just the statement that a solution to the K\"ahler-Ricci flow exists for some short time and is unique.  It then follows immediately that there must exist a solution on a maximal time interval.  The theorem does not tell us anything about the quantity $S$.  The point of Theorem \ref{thmmaximal}  is that $S=T$, where $T$ has the simple formulation (\ref{deT}) in terms of cohomology classes.

We now begin the proof of Theorem \ref{thmmaximal}.  By the discussion of Section \ref{sectionmax1}, we must have $S \le T$.  Since we wish to show that $T=S$, we will assume for a contradiction that $S<T$.

First we pick a smooth family of reference K\"ahler metrics.  Since $S<T$, the cohomology class $[\omega_0]-Sc_1(M)$ lies in the K\"ahler cone and hence contains a K\"ahler metric $\hat{\omega}_S$, say. We define $t\mapsto \hat{\omega}_t$ for $t \in [0,S]$ to be the linear path of metrics between $\omega_0$ and $\hat{\omega}_S$ (Figure \ref{fref}).  Namely:
\begin{equation} \label{hatot}
\hat{\omega}_t = \frac{1}{S} ((S-t)\omega_0 + t \hat{\omega}_S)= \omega_0 + t\chi \in [\omega_0] - tc_1(M),
\end{equation}
where we define 
\begin{equation} \label{chi}
\chi = \frac{1}{S} (\hat{\omega}_S -\omega_0) \in - c_1(M).
\end{equation}

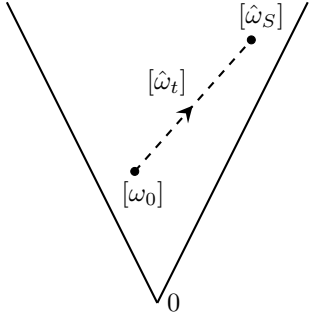
\begin{figure}[h!] 
\begin{tikzpicture}[scale=5, important line/.style={thick},
    dashed line/.style={dashed,->, thick},
        every node/.style={color=black} ]
        tik
\draw[important line] (0,0) coordinate (A) -- (.4,.8) coordinate (B) node[right, text width=5em] {};
\draw[important line] (-.4,.8) coordinate (B)  -- (0,0) coordinate (A) node[right, text width=5em]{$0$};

\draw[fill] (-.06,.35) circle (.3pt);
\draw[fill] (.25, .70) circle (.3pt);
\draw[directed, dashed, thick] (-.06, .35) coordinate (C) node[below,  text width=1em] {$[\omega_0]$} -- (.25,.70) coordinate (A) node[above, text width=1em]{$[\hat{\omega}_S]$}
node[midway, above, text width=3.5em] {$[\hat{\omega}_t]$};

\end{tikzpicture}    

\caption{The path of reference metrics $\hat{\omega}_t$ in the K\"ahler cone} \label{fref}
\end{figure}

We now make the key claim:  that
 the K\"ahler-Ricci flow (\ref{krf}) is equivalent to the parabolic complex Monge-Amp\`ere equation
\begin{equation} \label{ma}
\ddt{\varphi} = \log \frac{(\hat{\omega}_t + \ddbar \varphi)^n}{\Omega}, \quad \hat{\omega}_t + \ddbar \varphi>0, \quad \varphi|_{t=0} =0,
\end{equation}
where $\Omega$ is a volume form with 
\begin{equation} \label{ddbarOmega}
\ddbar \log \Omega = \chi, \quad \int_M \Omega = \int_M \omega_0^n,
\end{equation}
(recall Exercise \ref{exOmega}).  

First, why does there exist a volume form $\Omega$ satisfying (\ref{ddbarOmega})? 
Since $-\chi \in c_1(M)$ and $-\ddbar \log \omega_0^n \in c_1(M)$ we have
$$\chi = \ddbar \log \omega_0^n + \ddbar f= \ddbar \log (\omega_0^n e^f),$$
for some real valued function $f$.  Hence we may take $\Omega = \omega_0 e^{f+c}$ where $c$ is a constant chosen so that $\int_M \Omega = \int_M \omega_0^n$.

We now prove the claim.  Suppose that $\varphi$ solves (\ref{ma}) and set $\omega(t) = \hat{\omega}_t + \ddbar \varphi$.  Then
$$\ddt{} \omega = \chi + \ddbar \log \frac{\omega^n}{\Omega} = - \textrm{Ric}(\omega)$$
with $\omega|_{t=0}= \omega_0$ as required.  The other direction is left as an exercise.

\begin{exer} \label{ex31}
Finish the proof of the claim that (\ref{ma}) is equivalent to (\ref{krf}).
\end{exer}

To prove  Theorem \ref{thmmaximal} we will establish:

\begin{prop} \label{pestimates}  Let $\varphi=\varphi(t)$ solve (\ref{ma}) on $[0,S)$.  Then for each $k=0,1,2, \ldots$ there exists a positive constant $A_k$ such that on $[0,S)$,
$$\| \varphi \|_{C^k(M)} \le A_k, \quad \hat{\omega}_t + \ddbar \varphi \ge \frac{1}{A_0} \omega_0.$$
\end{prop}

The point here is that the bounds are independent of $t$. Recall that the $C^k$ norm of a function is defined by taking the sum of the $\sup$ norms of the $0$th through $k$th derivatives of the function (with respect to some fixed Riemannian metric).

Given Proposition \ref{pestimates}, we will complete the proof of Theorem \ref{thmmaximal}.  Since the bounds on $\varphi$ are independent of $t$ in $[0,S)$ we can apply the Arzel\`a-Ascoli Theorem to see that for a sequence of times $t_i \rightarrow S$,
$$\varphi(t_i) \rightarrow \varphi(S) \quad \textrm{in $C^{\infty}$ as } i \rightarrow \infty,$$
for a smooth function $\varphi(S)$ with $\omega(S):=\omega_0 + \ddbar \varphi(S) \ge \frac{1}{C_0} \omega_0 >0$.  Indeed, the Arzel\`a-Ascoli Theorem says in particular that if $\{ f_j \}$ is a sequence of functions bounded in $C^{k+1}(M)$ then there exists a convergence subsequence   in $C^k(M)$.  Using a diagonal argument, the $C^{\infty}$ estimates of Proposition \ref{pestimates} give $C^{\infty}$ convergence (i.e. convergence in every $C^k$ norm) of $\varphi(t_i)$ for a sequence of times $t_i \rightarrow S$.

In fact since $\dot{\varphi}$ is bounded, $\varphi(t)$ converges to a unique $\varphi(S)$ as $t\rightarrow S$ (see Exercise {\ref{ex45} below).  Now apply Hamilton's Theorem \ref{hamilton} to obtain a solution to the K\"ahler-Ricci flow starting at $\omega(S)$, for at least some short time $[0,\ve)$.  Putting the two solutions together we obtain a solution to the K\"ahler-Ricci flow starting at $\omega_0$ on the interval $[0, S+\ve)$. But this contradicts the maximality of $S$.

\begin{exer}  \label{ex45} For $t \in [0,T_0)$ (with $0< T_0 \le \infty$), 
let $f_t:M \rightarrow \mathbb{R}$ be a family of smooth functions which are uniformly bounded in $C^{\infty}$, independent of $t$.  
\begin{enumerate}
\item[(a)]  Suppose that for some function $f :M \rightarrow \mathbb{R}$, we have
$$
f_t \rightarrow f \  \textrm{pointwise on $M$, as } t \rightarrow T_0.
$$
Then $f_t$ converges in $C^{\infty}$ to $f$.  In particular, $f$ is smooth.
\item[(b)] Suppose that $T_0<\infty$ and $\partial{f}_t/\partial t$ is uniformly bounded.  Then show that there exists a unique smooth function $f:M \rightarrow \mathbb{R}$ such that $f_t$ converges in $C^{\infty}$ to $f$.
\end{enumerate}
\end{exer}

It remains then to prove Proposition \ref{pestimates}.

\section{Estimates on $\varphi$ and $\dot{\varphi}$}

In this and the next section, we will make use of the \emph{maximum principle} to prove Proposition \ref{pestimates}.  We use only a simple version of the maximum principle,  a consequence of elementary calculus, which can be stated as follows:

\bigskip
\noindent
{\it The maximum principle.} \ 
Let  $f=f(x,t)$ be a smooth function on $M \times [0,a]$ for $a>0$ and $M$ a compact manifold. Then $f$ achieves a global maximum at some point $(x_0, t_0) \in M \times [0,a]$.  At $x_0$ we have
$$\ddbar f (x_0, t_0) \le 0.$$
The sign of $\displaystyle{\ddt{f}(x_0, t_0)}$ depends on the value of $t_0$.  There are three cases (Figure \ref{ff}):
\begin{enumerate}
\item[(i)] If $t_0=0$ then $\displaystyle{\ddt{f}(x_0, t_0)\le 0}$.
\item[(ii)] If $t_0 \in (0,a)$ then $\displaystyle{\ddt{f}(x_0, t_0)=0}$.
\item[(iii)] If $t_0=a$ then $\displaystyle{\ddt{f}(x_0, t_0)\ge 0}$.
\end{enumerate}
In particular, if $t_0 \neq 0$ then $\displaystyle{\ddt{f}(x_0, t_0)\ge 0}$. 

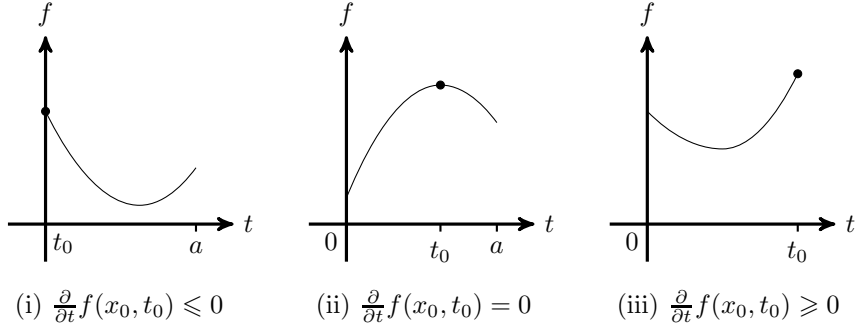
\begin{figure}[h!]
\begin{center}
\begin{tikzpicture}[scale=5, important line/.style={thick}, ax
    dashed line/.style={dashed,->, thick}, axis/.style={very thick, ->, >=stealth'},
        every node/.style={color=black} ]
        \draw[axis] (-0.1,0)  -- (.5,0) node(xline)[right]
        {$t$};
    \draw[axis] (0,-0.1) -- (0,.5) node(yline)[above] {$f$};
    \draw(0,.3) parabola bend (0.25, .05) (.4, .15);
    \draw[fill] (0,.3) circle (.3pt);
     \draw (0.1,0) node[right, below, text width=2.3em] {$t_0$};
 \draw[important line] (.4,0) -- (.4, -.02) node[below] {$a$};
 \draw (0.27,-.15) node[below, text width=10em] {(i) $\ddt{f}(x_0, t_0) \le 0$};
 
         \draw[axis] (.7,0)  -- (1.3,0) node(xline)[right]
        {$t$};
    \draw[axis] (0.8,-0.1) -- (.8,.5) node(yline)[above] {$f$};
    \draw(0.8,.07) parabola bend (1.05, .37) (1.2, .27);
    \draw[fill] ((1.05, .37) circle (.3pt);
         \draw (0.9,0) node[left, below, text width=4.5em] {$0$};
        \draw[important line] (1.05,0) -- (1.05, -.02) node[below] {$t_0$};
      \draw[important line] (1.2,0) -- (1.2, -.02) node[below] {$a$};
      \draw (1.07,-.15) node[below, text width=10em] {(ii) $\ddt{f}(x_0, t_0) = 0$};
     
              \draw[axis] (1.5,0)  -- (2.1,0) node(xline)[right]
        {$t$};
    \draw[axis] (1.6,-0.1) -- (1.6,.5) node(yline)[above] {$f$};
    \draw(1.6,.3) parabola bend (1.8, .2) (2.0, .4);
    \draw[fill] (2,.4) circle (.3pt);
              \draw (1.7,0) node[left, below, text width=4.5em] {$0$};
           \draw[important line] (2.0,0) -- (2.0, -.02) node[below] {$t_0$};
               \draw (1.87,-.15) node[below, text width=10em] {(iii) $\ddt{f}(x_0, t_0) \ge 0$};
 \end{tikzpicture}  
   
\end{center}
\caption{The time derivative at the maximum of $f$} \label{ff}
\end{figure}

Of course we can replace maximum by minimum, if we reverse all the inequalities.

We apply the maximum principle to prove estimates on $\varphi=\varphi(t)$ which we assume solves (\ref{ma}) on $[0,S)$.  The following two lemmas are due to Tian-Zhang \cite{TZha}.

\begin{lem} \label{lemmaphi} There exists a uniform constant $C$ so that on $M \times [0,S)$,
$$ |\varphi| \le C.$$
\end{lem}
\begin{proof}
Define a function $\psi = \varphi - At$ for a constant $A$ which we will specify later.  Compute
$$\ddt{\psi} = \log \frac{(\hat{\omega}_t+\ddbar \varphi)^n}{\Omega} - A =  \log \frac{(\hat{\omega}_t+\ddbar \psi)^n}{\Omega} - A.$$
At a point 
when $\ddbar \psi \le 0$ we have  $\hat{\omega}_t+\ddbar \psi \le \hat{\omega}_t$ and so
$$\log \frac{(\hat{\omega}_t+\ddbar \psi)^n}{\Omega} \le \log \frac{\hat{\omega}_t^n}{\Omega}.$$ 
We now pick
$$A = \sup_{M \times [0,S]} \log \frac{\hat{\omega}_t^n}{\Omega}+1.$$
Observe that $\hat{\omega}_t$ is a smooth family of K\"ahler metrics on $[0,S]$ (this uses the assumption that $S<T$) and hence $A$ is a uniform constant.    Then if $\ddbar \psi \le 0$, we have
$$\ddt{\psi} \le -1.$$
We can now use this to deduce that the maximum of $\psi$ must occur at $t=0$, which gives an upper bound for $\psi$ and hence for $\varphi$.

Let's make this last part more precise.  Fix $a \in (0,S)$.  Then $\psi$ is a smooth function on $M \times [0,a]$.  Suppose that $\psi$ achieves a maximum at $(x_0, t_0)$.  If $t_0>0$ then $\ddt{\psi} \ge 0$ at $(x_0, t_0)$, contradicting the inequality $\ddt{\psi} \le -1$ above.  Hence the maximum of $\psi$ on $M \times [0,a]$ is achieved at $t=0$.  Recalling that $\varphi|_{t=0}=0$, we obtain that $\psi \le 0$ on $M \times [0,a]$ and hence
$$\varphi \le At \le AS \quad \textrm{on } M \times [0,a].$$
Since $a$ was an arbitrary number in $(0,S)$, this gives the uniform upper bound $C=AS$ for $\varphi$ on $M \times [0,S)$.

The lower bound is left as an exercise.
\end{proof}

\begin{exer}  \label{ex32} Complete the proof of Lemma \ref{lemmaphi}.\end{exer}

Next we bound $\dot{\varphi}: = \frac{\partial \varphi}{\partial t}$, which is equivalent to a bound on the volume form of the evolving metric.  

\begin{lem} \label{lemdotphi}
There exists a uniform constant $C>0$ so that on $M \times [0,S)$,
$$ |\dot{\varphi}| \le C,$$
and  
$$C^{-1} \Omega \le \omega^n \le C \Omega.$$
\end{lem}
\begin{proof}  
For the lower bound of $\dot{\varphi}$ define 
$$Q = (S-t+\ve)\dot{\varphi} + \varphi +nt,$$
where $\ve>0$ is a positive constant to be determined later.  Differentiating (\ref{ma}) with respect to $t$ and  recalling (\ref{logdet}), (\ref{hatot}) and (\ref{chi}),
$$\ddt{\dot{\varphi}} = \tr{\omega}{\left(\ddt{} (\hat{\omega}_t + \ddbar \varphi) \right)} = \Delta \dot{\varphi} + \tr{\omega}{\chi},$$
where $\tr{}{}$ is defined by Exercise \ref{extr}.  Here, the Laplace operator $\Delta$ is defined by
$$\Delta f : = \tr{\omega} {(\ddbar f)} = g^{i\ov{j}} \partial_i \partial_{\ov{j}} f,$$
for a function $f$.
It then follows that
\[
\begin{split}
\left( \ddt{} - \Delta \right) Q = {} & (S-t+\ve) \tr{\omega}{\chi} - \dot{\varphi} + \dot{\varphi} - \Delta \varphi + n \\
={} & (S-t+\ve) \tr{\omega}{\chi} - \tr{\omega}{(\omega - \hat{\omega}_t)} + n \\
={} & \tr{\omega}{((S-t+\ve)\chi + \omega_0 + t\chi)} \\
= {} & \tr{\omega}{(\hat{\omega}_S+\ve \chi)} >0,
\end{split}
\]
if we choose $\ve>0$ small enough so that $\hat{\omega}_S + \ve \chi$ is positive definite (recall that $\hat{\omega}_S$ is a K\"ahler metric).
We now apply the maximum principle (strictly speaking, the minimum principle).  If $Q$ achieves a minimum on some compact time interval then at that minimum, we must have either $t_0=0$ or both $\ddt{Q} \le 0$ and $\Delta Q \ge 0$.  But the above inequality shows that the latter cannot occur and hence the minimum of $Q$ must occur at $t_0=0$.  It follows that for $t \in [0,S)$,
$$(S-t+\ve)\dot{\varphi} + \varphi + nt \ge -C,$$
and hence
$$\dot{\varphi} \ge  \frac{1}{\ve}\left( - C- nS - \sup_{M \times [0,S)} |\varphi|\right) := C',$$
since $|\varphi|$ is bounded by Lemma \ref{lemmaphi}.  The upper bound for $\dot{\varphi}$ is left as an exercise.

The estimate on the volume form follows immediately from the equation (\ref{ma}) and the bound on $\dot{\varphi}$.
\end{proof}

\begin{exer} \label{ex33}
Complete the proof of Lemma \ref{lemdotphi}.
\end{exer}

A remark about constants:  in what follows we will often use $C, C'$ etc to denote a uniform constant (the uniformity should be clear from the context) which may differ from line to line.

\section{Estimate on the metric}

We next bound the evolving metric $\omega=\omega(t)$.

\begin{lem} \label{lemtr}
There exists a unform constant $C$ such that on $M \times [0,S)$,
$$\emph{tr}_{\omega_0}{\, \omega} \le C.$$
\end{lem}

This result follows from the argument of Cao \cite{Cao}, and is a parabolic version of an estimate of Yau \cite{Y2} and Aubin \cite{A} (an alternative approach is to prove a parabolic Schwarz lemma \cite{Y1, SoT1}).
Note that once we have Lemma \ref{lemtr} together with the bound on the volume from Lemma \ref{lemdotphi}, we have
\begin{equation} \label{omegadb}
C^{-1} \omega_0 \le \omega \le C \omega_0,
\end{equation}
for a uniform positive constant $C>0$.  Indeed,  this follows from the next exercise:

\begin{exer} \label{exeq}
Let $\omega$ and $\omega_0$ be K\"ahler metrics.  Suppose there exists a uniform constant $C>0$ such that
$$\frac{1}{C} \omega_0^n \le \omega^n \le C \omega_0^n.$$
Then 
$$\tr{\omega_0}{\omega} \le C_1 \quad \Longleftrightarrow \quad \tr{\omega}{\omega_0} \le C_2 \quad \Longleftrightarrow \quad \frac{1}{C_3} \omega_0 \le \omega \le C_3 \omega_0,$$
where each constant $C_i>0$ can depend on $C$ and the constant $C_j$ in the bound that is being assumed.  So for example, in proving the first implication $\Longrightarrow$, $C_2$ may depend on $C$ and $C_1$.
{\it Hint:  choose coordinates so that $g_0$ is the identity and $g$ is diagonal.}
\end{exer}

\begin{proof}[Proof of Lemma \ref{lemtr}]
We require two key calculations.  The first is
\begin{equation} \label{calc1}
\left( \ddt{} - \Delta \right) \log \tr{\omega_0}{\omega} = - \frac{1}{\tr{\omega_0}{\omega}} g^{i\ov{j}} R_{\, i\ov{j}}^{0 \ \, \ov{\ell} k} g_{k\ov{\ell}} + (\dagger),
\end{equation}
where
\begin{equation} \label{daggerdef}
(\dagger) = - \frac{g_0^{k\ov{\ell}} g^{i\ov{j}} g^{p\ov{q}} \nabla^0_i g_{k\ov{q}} \nabla^0_{\ov{j}} g_{p\ov{\ell}}}{\tr{\omega_0}{\omega}} + \frac{| \partial \tr{\omega_0}{\omega}|^2_g}{(\tr{\omega_0}{\omega})^2}.
\end{equation}
Here, we are using $R_{\, i\ov{j}}^{0 \ \, \ov{\ell} k}$ and $\nabla^0$ to denote the curvature tensor and covariant derivative with respect to $g_0$.
The second calculation is the inequality:
\begin{equation} \label{daggerex}
(\dagger) \le 0,
\end{equation}
which we will leave as an exercise (see below).

Given these two calculations, we can easily complete the proof of the lemma.  We define a quantity
$$Q= \log \tr{\omega_0}{\omega} - A \varphi,$$
where $A \ge 1$ is a constant to be determined soon.  We note that by an elementary local calculation,
$$| g^{i\ov{j}} R_{\, i\ov{j}}^{0 \ \, \ov{\ell} k} g_{k\ov{\ell}} | \le C_0 (\tr{\omega}{\omega_0})(\tr{\omega_0}{\omega}),$$
for a uniform constant $C_0$ which depends only on the curvature of $g_0$.  Then from (\ref{calc1}), (\ref{daggerdef}), (\ref{daggerex}) we have
\[
\begin{split}
\left( \ddt{} - \Delta \right) Q \le {} & C_0 \tr{\omega}{\omega_0} - A \dot{\varphi} + A \Delta \varphi \\
= {} & C_0 \tr{\omega}{\omega_0} - A\dot{\varphi} + A \tr{\omega}{(\omega- \hat{\omega}_t)} \\
\le {} & \tr{\omega}{(C_0 \omega_0 - A \hat{\omega}_t)} + C'A,
\end{split}
\]
where we have used the fact that $\dot{\varphi}$ is bounded.  Now choose $A$ large enough so that
$$C_0 \omega_0 - A\hat{\omega}_t \le - \omega_0.$$
Since the family of reference metrics $\hat{\omega}_t$ is uniformly bounded, the constant $A$ is uniform.  Then at a point $(x_0, t_0)$ where  $Q$ achieves a maximum, assuming that $t_0 \neq 0$, we have by the maximum principle,
$$0 \le - \tr{\omega}{\omega_0} + C'A,$$
and hence $\tr{\omega}{\omega_0} \le C$ at $(x_0, t_0)$.  From Exercise \ref{exeq} we obtain at $(x_0, t_0)$,
$$\tr{\omega_0}{\omega} \le C.$$
But since $\varphi$ is uniformly bounded, we see that $Q$ is bounded from above at $(x_0, t_0)$ and hence 
$$\log \tr{\omega_0}{\omega} - A\varphi \le C \qquad \textrm{on } M \times [0,S).$$
Note that the case when $t_0=0$ is trivial.
Again using the fact that $\varphi$ is uniformly bounded we obtain
$$\tr{\omega_0}{\omega} \le C \qquad \textrm{on } M \times [0,S),$$
as required.

It remains to establish (\ref{calc1}).  Since it is an inequality of tensors, we are free to choose any coordinate system centered at a fixed point $x$, say.  We choose holomorphic normal coordinates for $g_0$ as provided by Lemma \ref{lemnormal}.  First compute using (\ref{krf}),
\begin{equation} \label{ddttr}
\ddt{} \log \tr{\omega_0}{\omega} = \frac{1}{\tr{\omega_0}{\omega}} \tr{\omega_0}{\left( \ddt{} \omega \right)} = - \frac{1}{\tr{\omega_0}{\omega}} g_0^{i\ov{j}} R_{i\ov{j}}.
\end{equation}
Next, using the fact that the first derivatives of $g_0$ vanish at $x$,
\[
\Delta \tr{\omega_0}{\omega} = g^{i\ov{j}} \partial_i \partial_{\ov{j}} (g_0^{k\ov{\ell}} g_{k\ov{\ell}}) = g^{i\ov{j}} (\partial_i \partial_{\ov{j}} g_0^{k\ov{\ell}}) g_{k\ov{\ell}} + g^{i\ov{j}} g_0^{k\ov{\ell}} \partial_i \partial_{\ov{j}} g_{k\ov{\ell}}.
\]
But, applying the formula (\ref{Rf}) with the metric $g_0$,
\[
\partial_i \partial_{\ov{j}} g_0^{k\ov{\ell}} = - \partial_i ( g_0^{p\ov{\ell}} g_0^{k\ov{q}} \partial_{\ov{j}} (g_0)_{p\ov{q}}) = - g_0^{p\ov{\ell}} g_0^{k\ov{q}} \partial_i \partial_{\ov{j}} (g_0)_{p\ov{q}} = R^{0 \ \, \ov{\ell} k}_{\, i\ov{j}}.
\]
Combining the above and using (\ref{Rf}) with the metric $g$,
\[
\begin{split}
\Delta \tr{\omega_0}{\omega} = {} & g^{i\ov{j}} R^{0 \ \, \ov{\ell} k}_{\, i\ov{j}} g_{k\ov{\ell}} - g^{i\ov{j}} g_0^{k\ov{\ell}} R_{i\ov{j} k\ov{\ell}} + g^{i\ov{j}} g_0^{k\ov{\ell}} g^{p\ov{q}} \partial_i g_{k\ov{q}} \partial_{\ov{j}} g_{p\ov{\ell}} \\
= {} & g^{i\ov{j}} R^{0 \ \, \ov{\ell} k}_{\, i\ov{j}} g_{k\ov{\ell}} - g_0^{k\ov{\ell}} R_{k\ov{\ell}} + g^{i\ov{j}} g_0^{k\ov{\ell}} g^{p\ov{q}} \nabla^0_i g_{k\ov{q}} \nabla^0_{\ov{j}} g_{p\ov{\ell}},
\end{split}
\]
where for the last equality we use $g^{i\ov{j}} R_{i\ov{j} k\ov{\ell}} = R_{k\ov{\ell}}$ and, at the point $x$, $\nabla^0_i = \partial_i$.  Then from (\ref{ddttr}),
\[
\begin{split}
\lefteqn{\left(\ddt{} - \Delta \right) \log \tr{\omega_0}{\omega} } \\= {} & \ddt{} \log \tr{\omega_0}{\omega} - \frac{\Delta \tr{\omega_0}{\omega}}{\tr{\omega_0}{\omega}} + \frac{|\partial \tr{\omega_0}{\omega} |^2_g}{(\tr{\omega_0}{\omega})^2} \\
={} &  - \frac{1}{\tr{\omega_0}{\omega}} g_0^{i\ov{j}} R_{i\ov{j}} - \frac{1}{\tr{\omega_0}{\omega}} \left( g^{i\ov{j}} R^{0 \ \, \ov{\ell} k}_{\, i\ov{j}} g_{k\ov{\ell}} - g_0^{k\ov{\ell}} R_{k\ov{\ell}} + g_0^{k\ov{\ell}} g^{i\ov{j}} g^{p\ov{q}} \nabla^0_i g_{k\ov{q}} \nabla^0_{\ov{j}} g_{p\ov{\ell}} \right) \\
& + \frac{|\partial \tr{\omega_0}{\omega} |^2_g}{(\tr{\omega_0}{\omega})^2} \\
= {} & - \frac{1}{\tr{\omega_0}{\omega}} g^{i\ov{j}} R_{\, i\ov{j}}^{0 \ \, \ov{\ell} k} g_{k\ov{\ell}} + (\dagger),
\end{split}
\]
as required.
\end{proof}

In the proof, we made use of:

\begin{exer}  \label{ex35} Show that $(\dagger) \le 0$ as follows.  Define
$$B_{i\ov{j}k} = \nabla^0_i g_{k\ov{j}} - \frac{\partial_k ( \tr{\omega_0}{\omega})}{\tr{\omega_0}{\omega}} g_{i\ov{j}},$$
and then show that
$$0 \le g_0^{i\ov{q}} g^{p\ov{j}} g^{k\ov{\ell}} B_{i\ov{j}k} \ov{B_{q\ov{p} \ell}} = g_0^{i\ov{q}} g^{p\ov{j}} g^{k\ov{\ell}} \nabla^0_k g_{i\ov{j}} \nabla^0_{\ov{\ell}} g_{p\ov{q}} - \frac{| \partial \tr{\omega_0}{\omega}|^2_g}{\tr{\omega_0}{\omega}}.$$
\end{exer}

\section{Higher order estimates} \label{secthigher}

We now complete the proof of Proposition \ref{pestimates} and hence Theorem \ref{thmmaximal}.
Given Lemma \ref{lemtr}  the proof of Proposition \ref{pestimates} follows from fairly standard parabolic theory, which we will quote without proof.  We already have from (\ref{omegadb}) the estimate
$$\omega= \hat{\omega}_t + \ddbar \varphi \ge \frac{1}{C_0} \omega_0,$$
and we also have immediately
$$| \ddbar \varphi |_{g_0} \le C.$$
We can then apply the parabolic Evans-Krylov theory \cite{E, K} (for a proof in the complex setting see \cite{Gi}) 
to obtain the parabolic Schauder estimate 
$$\| \varphi \|_{C^{2+\alpha, 1+\alpha/2}} \le C,$$
where the $1+\alpha/2$ refers to the derivative in the $t$ direction.  An alternative to using Evans-Krylov is to prove ``Calabi'' and curvature estimates using the maximum principle  \cite{C1, Cao, Cha, PSS, PSSW3, ShW}. The higher order estimates for $\varphi$ follow from a simple ``bootstrap'' argument.  Indeed, we can apply the differential operator $L= \frac{\partial}{\partial x^k}$ say to (\ref{ma}) to obtain, locally,
$$\ddt{} L(\varphi) = g^{i\ov{j}} \partial_i \partial_{\ov{j}} L(\varphi) + g^{i\ov{j}} L((\hat{g}_t)_{i\ov{j}}) - L \left(\log (\Omega) \right).$$
This is a linear parabolic equation in $L(\varphi)$ with coefficients in $C^{\alpha, \alpha/2}$.  The standard parabolic estimates (see for example \cite{Lie}) give a $C^{2+\alpha, 1+\alpha/2}$ bound for $L(\varphi)$.  Since $L$ was any first order operator with constant coefficients, we obtain a bound for $\varphi$ in $C^{3+\alpha, 1+\alpha/2}$.  This gives higher regularity for the coefficients of the parabolic equation above, and we repeat the argument to obtain higher regularity for $\varphi$.  Thus we obtain bounds for $\varphi$ in all derivatives.

\chapter{Convergence results}

In this lecture we describe convergence results for the K\"ahler-Ricci flow in the case when the first Chern class of $M$ is negative or zero.

\section{Negative first Chern class}

Let $(M, \omega_0)$ be a compact K\"ahler manifold with $c_1(M)<0$.  Observe that 
$$T= \sup \{ t>0 \ | \ [\omega_0] - t c_1(M) >0 \} = \infty$$
and hence there is a solution to the K\"ahler-Ricci flow for all time.  Note however that the K\"ahler class $[\omega(t)]$ becomes unbounded as $t \rightarrow \infty$,  so there is no chance that the flow can converge to a smooth K\"ahler metric.  For this reason, we rescale the flow as follows.

Suppose that $\tilde{\omega}(s)$ solves the K\"ahler-Ricci flow $\frac{\partial}{\partial s} \tilde{\omega}(s) = - \textrm{Ric}(\tilde{\omega}(s))$ on $[0, \infty)$  with $\tilde{\omega}(0)=\omega_0$.  Define 
$$\omega(t) = \frac{1}{s+1} \tilde{\omega}(s), \quad t = \log (s+1).$$
Then $\omega(0) = \omega_0$ and
$$\ddt{} \omega(t) = -\frac{1}{(s+1)^2} \frac{ds}{dt} \, \tilde{\omega}(s) - \frac{1}{s+1} \frac{ds}{dt} \frac{\partial}{\partial s} \tilde{\omega}(s) = - \omega(t) - \textrm{Ric}(\omega(t)),$$
where we have used (\ref{scaleRic}) for the last equality. We call this equation,
\begin{equation} \label{nkrf}
\ddt{} \omega = - \textrm{Ric}(\omega) - \omega, \qquad \omega|_{t=0} = \omega_0,
\end{equation}
the \emph{normalized K\"ahler-Ricci flow}.   Whenever $c_1(M)<0$, we have a solution to this flow for all time.  Moreover, the K\"ahler class $[\omega(t)]$ satisfies the ordinary differential equation
$$\frac{d}{dt} [\omega(t)] = - c_1(M) - [\omega(t)], \quad [\omega(t)]=[\omega_0],$$
whose solution is
\begin{equation} \label{kc}
[\omega(t)] = e^{-t} [\omega_0] + (1-e^{-t}) [-c_1(M)].
\end{equation}
Thus $[\omega(t)]$ moves in a straight line from $[\omega_0]$ to the K\"ahler class $-c_1(M)$.

In this section, we prove the following theorem.

\begin{thm} \label{thmcao}
Suppose that $c_1(M)<0$.  Then starting at any K\"ahler metric $\omega_0$ the solution to the normalized K\"ahler-Ricci flow (\ref{nkrf}) exists for all time and converges in $C^{\infty}$ to a K\"ahler-Einstein metric $\omega_{\emph{KE}}$, which satisfies
\begin{equation} \label{keeqn}
\emph{Ric}(\omega_{\emph{KE}}) = - \omega_{\emph{KE}}.
\end{equation}
Moreover, $\omega_{\emph{KE}}$ is the unique K\"ahler metric solving (\ref{keeqn}).
\end{thm}

The existence of a solution to (\ref{keeqn}) was proved independently by Yau \cite{Y2} and Aubin \cite{A} in the 1970s, and the uniqueness was shown earlier by Calabi \cite{C0}.  H.-D. Cao \cite{Cao} proved that the K\"ahler-Ricci flow starting at any K\"ahler metric $\omega_0$ in the class $-c_1(M)$ converges to a solution of (\ref{keeqn}), giving a parabolic proof of the result of Yau and Aubin.  The theorem we state here is slightly more general than Cao's result, since we allow $\omega_0$ to lie in an arbitrary K\"ahler class.  Theorem \ref{thmcao} follows from the work of Tsuji \cite{Ts1} and Tian-Zhang  \cite{TZha} who dealt more generally with  K\"ahler manifolds with $-c_1(M)$ being positive in a weaker sense (more precisely, when $M$ has big and nef canonical bundle, also known as a smooth minimal model of general type).

As in the previous section, we reduce the flow equation (\ref{nkrf}) to a parabolic complex Monge-Amp\`ere equation.  We first need to pick reference K\"ahler metrics in the cohomology class 
$[\omega(t)]$. 
Since $c_1(M)<0$ there exists a K\"ahler metric $\hat{\omega}_{\infty}$ in $-c_1(M)$, which we fix once and for all.  Define
$$\hat{\omega}_t = e^{-t} \omega_0 + (1-e^{-t})\hat{\omega}_{\infty},$$
which clearly lies in $[\omega(t)]$ given by (\ref{kc}).  Next, we fix a volume form $\Omega$ satisfying
\begin{equation} \label{Omega2}
\ddbar \log \Omega = \hat{\omega}_{\infty}, \quad \int_M \Omega = \int_M \omega_0^n.
\end{equation}

\begin{exer} \label{ex41}
Show that there exists a volume form $\Omega$ satisfying (\ref{Omega2}).
\end{exer}

We can then rewrite the normalized K\"ahler-Ricci flow as the parabolic complex Monge-Amp\`ere equation:

\begin{equation} \label{nma}
\ddt{\varphi} = \log \frac{(\hat{\omega}_t + \ddbar \varphi)^n}{\Omega} - \varphi, \quad \hat{\omega}_t + \ddbar \varphi>0, \quad \varphi|_{t=0} =0,
\end{equation}

\begin{exer} \label{ex42}
Show that (\ref{nkrf}) is equivalent to (\ref{nma}).
\end{exer}

From now on let $\varphi(t)$ solve (\ref{nma}).  We wish to prove estimates for $\varphi$ which are independent of $t$.  

\begin{prop} \label{pestimates2} Let $\varphi=\varphi(t)$ solve (\ref{nma}) for $t\in [0,\infty)$.  Then for each $k=0,1,2,\ldots$ there exists a positive constant $A_k$ such that on $[0,\infty)$,
$$\| \varphi \|_{C^k(M)} \le A_k, \quad \hat{\omega}_t + \ddbar \varphi \ge \frac{1}{A_0} \omega_0.$$
\end{prop}

In the same way that we proved Proposition \ref{pestimates}, we begin 
by establishing estimates on $\varphi$ and $\dot{\varphi}$. In this case we can prove stronger estimates which improve as $t\rightarrow \infty$.

\begin{lem}  \label{c1lem} There exists a uniform constant $C>0$ such that on $M \times [0,\infty)$, 
\begin{enumerate}
\item[(i)] $\displaystyle{| \varphi(t) | \le C}$.
\item[(ii)] $\displaystyle{ | \dot{\varphi}(t)| \le C(t+1) e^{-t}}$.
\item[(iii)] There exists a continuous function $\varphi_{\infty}$ on $M$ such that
$$| \varphi(t) - \varphi_{\infty}| \le Ce^{-t/2}.$$
\item[(iv)] $\displaystyle{C^{-1} \Omega \le \omega^n \le C \Omega}.$
\end{enumerate}
\end{lem}
\begin{proof}
Part (i) is left as an exercise.
For (ii), we use an argument of Tian-Zhang \cite{TZha}.  Compute
\[
\begin{split}
\left( \ddt{} - \Delta \right) \dot{\varphi} = {} & \tr{\omega}{\left( \ddt{} \hat{\omega}_t \right)} - 
\dot{\varphi} = \tr{\omega}{( -e^{-t} \omega_0 + e^{-t} \hat{\omega}_{\infty} ) - \dot{\varphi}}\\
\left( \ddt{}  - \Delta \right) \varphi = {} & \dot{\varphi} - \tr{\omega}{(\omega- \hat{\omega}_t)} = \dot{\varphi} - n + \tr{\omega}{(e^{-t}\omega_0 + (1-e^{-t}) \hat{\omega}_{\infty})}.
\end{split}
\]
Then we have
\begin{equation} \label{lem1}
\left( \ddt{} - \Delta \right) ( e^t \dot{\varphi}) =  \tr{\omega}{(-\omega_0 +\hat{\omega}_{\infty})} 
\end{equation}
and
\begin{equation} \label{lem2}
\left( \ddt{} - \Delta \right) (\dot{\varphi} + \varphi + nt) =  \tr{\omega}{\hat{\omega}_{\infty}}.
\end{equation}
Subtracting (\ref{lem2}) from (\ref{lem1}), we obtain
$$\left( \ddt{} - \Delta \right) \left( ( e^t - 1)\dot{\varphi} - \varphi -nt \right) = - \tr{\omega}{\omega_0} <0.$$
It then follows from the maximum principle that
$$( e^t - 1)\dot{\varphi} - \varphi -nt  \le 0,$$ 
and since $\varphi$ is bounded by (i), this gives the upper bound $\dot{\varphi} \le C(t+1)e^{-t}$.

For the lower bound of $\dot{\varphi}$ we add a large multiple of (\ref{lem2}) to (\ref{lem1}),
$$\left( \ddt{} - \Delta \right) \left( (e^t +A) \dot{\varphi} + A\varphi + Ant\right) = \tr{\omega}{(-\omega_0 +\hat{\omega}_{\infty} +A\hat{\omega}_{\infty})}>0$$
where $A$ is a constant chosen so that $A \hat{\omega}_{\infty} \ge \omega_0$.  Then 
$$\dot{\varphi} \ge - C(1+t) e^{-t}$$
follows from the minimum principle.

For (iii), compute for $s>t$ and $x \in M$,
\begin{equation} \label{pi}
\begin{split}
| \varphi(x,s) - \varphi(x,t)| =  {}&  \left| \int_t^s \dot{\varphi}(x,u)du \right| 
\le  \int_t^s |\dot{\varphi}(x,u) |du \\ 
\le {} & C \int_t^s  e^{-u/2} du = 2C(e^{-t/2} - e^{-s/2}),
\end{split}
\end{equation}
since from (ii) we have $|\dot{\varphi}|\le Ce^{-t/2}$.  Then from (\ref{pi}), $\varphi(t)$ converges uniformly to a continuous function $\varphi_{\infty}$.  Taking the limit in (\ref{pi}) as $s\rightarrow \infty$ gives (iii).

Part (iv) follows from (i) and (ii).
\end{proof}

\begin{exer} \label{ex43}
Prove part (i) of Lemma \ref{c1lem}.
\end{exer}

Next we prove a bound on the metric.

\begin{lem}  \label{lemtr2} There exists a uniform constant $C$ such that on $M\times[0,\infty)$,
$$C^{-1} \omega_0 \le \omega(t) \le C \omega_0.$$
\end{lem}
\begin{proof}
The proof is similar to that of Lemma \ref{lemtr}, so we will be brief. First, to bound the trace of $\omega$ with respect to $\omega_0$, we claim that
\begin{equation} \label{c1ex}
\left( \ddt{} - \Delta \right) \log \tr{\omega_0}{\omega} \le C_0 \tr{\omega}{\omega_0} -1,
\end{equation}
for a uniform $C_0$.  This calculation is left as an exercise. 

Now define $Q= \log \tr{\omega_0}{\omega} -A \varphi$, for $A$ to be determined.  Calculate
$$\left( \ddt{} - \Delta \right) Q \le C_0 \tr{\omega}{\omega_0} - 1 - A \dot{\varphi} + A \tr{\omega}{(\omega- \hat{\omega}_t)}.$$
Choose $A$ sufficiently large so that $A\hat{\omega}_t \ge (C_0+1) \omega_0$, which we can do so since the metrics $\hat{\omega}_t$ are uniformly bounded as $t \rightarrow \infty$.  Hence
$$\tr{\omega}{\omega_0} \le C$$
at the maximum of $Q$ (if occuring at $t_0>0$), using the fact $\dot{\varphi}$ is bounded by part (ii) of Lemma \ref{c1lem}. 
 From 
Exercise \ref{exeq} and Lemma \ref{c1lem} again, we see that $\tr{\omega_0}{\omega}$ and hence $Q$ is bounded from above at the maximum of $Q$.  The lemma follows by using Exercise \ref{exeq} once more.
\end{proof}

\begin{exer} \label{ex44}
Prove the inequality (\ref{c1ex}).
\end{exer}

We now have everything we need to obtain the $C^{\infty}$ estimates:

\begin{proof}[Proof of Proposition \ref{pestimates2}]  
The argument follows in exactly the same way as in Section \ref{secthigher}.  The only difference is that we are dealing with the normalized K\"ahler-Ricci flow, but this only adds a harmless term.
\end{proof}

To finish the proof of Theorem \ref{thmcao}, we need to prove convergence to a unique K\"ahler-Einstein metric.  From part (iii) of Lemma \ref{c1lem}:
$$| \varphi(t) - \varphi_{\infty} | \le Ce^{-t/2},$$
we know that $\varphi(t)$ converges uniformly exponentially fast to $\varphi_{\infty}$.
But   since we have $C^{\infty}$ estimates from Proposition \ref{pestimates2} we can apply the Exercise \ref{ex45} to see that $\varphi(t)$ converges to $\varphi_{\infty}$ in $C^{\infty}$, and in particular $\varphi_{\infty}$ is smooth.

Next, we apply part (ii) of Lemma \ref{c1lem}:
$$| \dot{\varphi}| \le C(t+1)e^{-t},$$
to see that $\dot{\varphi}$ converges in $C^{\infty}$ to $0$.  Then taking the limit as $t \rightarrow \infty$ of (\ref{nma}) we obtain
$$0 = \log \frac{(\hat{\omega}_{\infty} + \ddbar \varphi_{\infty})^n}{\Omega}- \varphi_{\infty}.$$
Taking $\ddbar$ of both sides of this equation and recalling (\ref{Omega2}), we have
$$\textrm{Ric}(\hat{\omega}_{\infty} + \ddbar \varphi_{\infty}) + \hat{\omega}_{\infty} + \ddbar \varphi_{\infty} =0,$$
or, in other words,
$\hat{\omega}_{\infty} + \ddbar \varphi$ satisfies the K\"ahler-Einstein equation (\ref{keeqn}).

It remains to prove the uniqueness of solutions to the K\"ahler-Einstein equation.  Suppose that $\oke$ and $\oke'$ are two solutions of (\ref{keeqn}).  Then $\oke, \oke'$ both lie in $-c_1(M)$ and so by the $\partial \ov{\partial}$ Lemma we can write
$$\oke' = \oke + \ddbar \psi$$
for some function $\psi$.  We have
$$\textrm{Ric}(\oke') = -\oke' = -\oke - \ddbar \psi = \textrm{Ric}(\oke) - \ddbar \psi,$$
and hence
$$- \ddbar \log \frac{\oke'^n}{\oke^n} = -\ddbar \psi.$$
Applying Exercise \ref{exerddbar}, we have
$$\log \frac{( \oke + \ddbar \psi)^n}{\oke^n} = \psi + C,$$
for some constant $C$.  We now apply the maximum principle to $\psi$ (a simpler version of the maximum principle, where the function has no dependence on $t$).  At the maximum of $\psi+C$, we have $\ddbar \psi \le 0$ and so at this point $\psi+C \le 0$.   By considering similarly the minimum of $\psi+C$ we obtain $\psi+C \ge 0$ and hence $\psi$ is constant and
 $\oke=\oke'$.
 
 This completes the proof of Theorem \ref{thmcao}.

\section{Zero first Chern class}

We now consider the case when $c_1(M)=0$.   Fix any K\"ahler metric $\omega_0$.  We have
$$T = \sup \{ t>0 \ | \ [\omega_0] - tc_1(M) >0 \} =\infty,$$
and so a solution to the K\"ahler-Ricci flow exists for all time.  The K\"ahler class $[\omega(t)]$ does not move, and so unlike the case of $c_1(M)<0$  there is no need to rescale the flow.

The behavior of the K\"ahler-Ricci flow is given by the following theorem:

\begin{thm} \label{thmcao2}
Suppose $c_1(M)=0$.  Then starting at any K\"ahler metric $\omega_0$, the K\"ahler-Ricci flow (\ref{krf}) exists for all time and converges in $C^{\infty}$ to a K\"ahler-Einstein metric $\omega_{\emph{KE}}$ satisfying
\begin{equation} \label{krflat}
\emph{Ric}(\omega_{\emph{KE}}) =0.
\end{equation}
Moreover, $\omega_{\emph{KE}}$ is the unique K\"ahler metric in $[\omega_0]$ satisfying (\ref{krflat}).
\end{thm}

The existence of a K\"ahler metric $\oke$ in each K\"ahler class satisfying (\ref{krflat}) is due to Yau, and the uniqueness part was already established by Calabi \cite{C0}.  H.-D. Cao \cite{Cao} proved Theorem \ref{thmcao2}, making use of Yau's $L^{\infty}$ estimate for the complex Monge-Amp\`ere equation \cite{Y2}:

\begin{thm} \label{thmyau}
Let $(M, \omega_0)$ be a compact K\"ahler manifold and let $F:M \rightarrow \mathbb{R}$ be a smooth function.  Suppose that $\theta$ satisfies the complex Monge-Amp\`ere equation
$$(\omega_0 + \ddbar \theta)^n = e^F \omega_0^n, \quad \omega_0 + \ddbar \theta>0.$$
Then 
$$\emph{osc}(\theta) := \sup_M \theta - \inf_M \theta \le C,$$
for a constant depending only on $(M, \omega_0)$ and $\sup_M F$.
\end{thm}

Note that if $\theta$ in this theorem is normalized by $\int_M \theta \, \omega^n=0$, say, then the conclusion of the theorem is that $\| \theta \|_{L^{\infty}} = \sup_M | \theta |\le C.$
We will omit the proof of Yau's Theorem \ref{thmyau} and proceed to prove Theorem \ref{thmcao2}.

We first prove the uniqueness part of Theorem \ref{thmcao2}.  Suppose $\oke$ and $\oke'$ are solutions of (\ref{krflat}) in the same K\"ahler class.  Then we can write $\oke' = \oke + \ddbar \psi$ for some smooth function $\psi$.  Since $\Ric(\oke) = \Ric(\oke')=0$ we have
$$\ddbar \log \frac{\oke'^n}{\oke^n}=0,$$
and so $\log (\oke'^n/\oke^n)$ is a constant.  Since the integral of $\oke^n$ is the same as the integral of $\oke'^n$ we have $\oke'^n = \oke^n$.   Then using Stokes' Theorem,
\[
\begin{split}
0 = \int_M \psi( \oke^n - \oke'^n) = {} & \int_M \psi (\oke - \oke') \wedge \sum_{i=0}^{n-1} \oke^i \wedge \oke'^{n-1-i}. \\
= {} &- \int_M \psi \ddbar \psi \wedge \sum_{i=0}^{n-1} \oke^i \wedge \oke'^{n-1-i} \\ 
= {} & \int_M \sqrt{-1} \partial \psi \wedge \ov{\partial} \psi \wedge \sum_{i=0}^{n-1} \oke^i \wedge \oke'^{n-1-i} \\
\ge {} & \int_M \sqrt{-1} \partial \psi \wedge \ov{\partial} \psi \wedge \oke^{n-1} = \frac{1}{n} \int_M | \partial \psi|^2_{\oke} \oke^n,
\end{split}
\]
where for the last inequality we are using the fact that all the terms that we are throwing away are nonnegative.  Hence
$\psi$ must be a constant and $\oke=\oke'$.

We now return to the K\"ahler-Ricci flow, and as usual  reduce the flow equation (\ref{krf}) to a parabolic complex Monge-Amp\`ere equation. The K\"ahler class does not change along the flow so we can choose $\omega_0$ as a reference metric.
  Since $c_1(M)=0$, there exists a volume form $\Omega$ with 
$$\ddbar \log \Omega =0 \quad \textrm{and} \quad \int_M \Omega = \int_M \omega_0^n.$$
Then the K\"ahler-Ricci flow is equivalent to:
\begin{equation} \label{may}
\ddt{\varphi} = \log \frac{(\omega_0 + \ddbar \varphi)^n}{\Omega}, \quad  \omega_0 + \ddbar \varphi>0, \quad \varphi|_{t=0} =0.
\end{equation}
From now on, let $\varphi=\varphi(t)$ solve this equation (\ref{may}).
We have the following lemma.

\begin{lem} \label{phiagain2}
There exists a uniform constant $C>0$ such that on $M \times [0,\infty)$,
\begin{enumerate}
\item[(i)] $\displaystyle{| \dot{\varphi}| \le C}.$
\item[(ii)] $\displaystyle{C^{-1} \Omega \le \omega^n \le C\Omega}$.
\item[(iii)] $\displaystyle{\emph{osc}(\varphi) \le C.}$
\end{enumerate}
\end{lem}
\begin{proof}
For (i), compute
$$\ddt{} \dot{\varphi} = \Delta \dot{\varphi}.$$
Then the bound on $\dot{\varphi}$ follows from the maximum principle (see the next exercise).  
Part (ii) follows from (i).  Part (iii) also follows from (i) together with Theorem \ref{thmyau}.
\end{proof}

We made use of:

\begin{exer} \label{ex46}
On a compact K\"ahler manifold $M$, let $f=f(x,t)$ satisfy the \emph{heat equation}
$$\ddt{} f = \Delta f, \quad f|_{t=0} = f_0,$$
where $\Delta$ is the Laplacian associated to $g(t)$, an arbitrary family of K\"ahler metrics. Then show that on $M \times [0,\infty)$,
$$|f| \le \sup_M |f_0|.$$
\emph{Hint:  consider $f\pm \ve t$.}
\end{exer}

Next we prove an estimate on the evolving metric $\omega(t)$.  Note that there is a complication here which did not appear in the case of $c_1(M)<0$ nor in the proof of Theorem \ref{thmmaximal}:  it arises because we only have a bound on the oscillation $\textrm{osc}(\varphi)$ and not on $| \varphi |$.

\begin{lem} \label{yso}
There exists a uniform constant $C$ such that on $M\times[0,\infty)$,
$$C^{-1} \omega_0 \le \omega(t) \le C \omega_0.$$
\end{lem}
\begin{proof}
We claim that there exist uniform constants $C$ and $A$ such that for $(x,t) \in M \times [0,\infty)$,
\begin{equation} \label{claimA}
(\tr{\omega_0}{\omega})(x,t) \le C \exp{\left( A\left( \varphi(x,t) - \inf_{M \times [0,t]} \varphi \right) \right)}.
\end{equation}
To prove (\ref{claimA}), apply the maximum principle to the quantity $Q=\log \tr{\omega_0}{\omega} - A \varphi$ for a large constant $A$, as in the proofs of Lemmas \ref{lemtr} and \ref{lemtr2}.  The details are left as the next exercise.

Now define $$\tilde{\varphi} = \varphi - \frac{1}{V} \int_M \varphi \, \Omega, \quad \textrm{where } V= \int_M \Omega = \int_M \omega^n.$$
Since the oscillation of $\varphi$ is bounded, $| \tilde{\varphi}|$ is uniformly bounded.  From (\ref{claimA}), we have
\[
\begin{split}
(\tr{\omega_0}{\omega})(x,t) \le {} & C\exp{\left(A \left( \tilde{\varphi}(x,t) + \frac{1}{V} \int_M \varphi(t) \Omega - \inf_{M\times [0,t]} \tilde{\varphi} - \inf_{[0,t]} \frac{1}{V} \int_M \varphi \, \Omega \right) \right)} \\
\le {} & C' \exp{\left( \frac{A}{V}\left( \int_M  \varphi(t) \Omega - \inf_{[0,t]} \int_M \varphi \, \Omega \right) \right) },
\end{split}
\]
where for the last inequality we used the bound on $|\tilde{\varphi}|$.  But,
$$\frac{d}{dt} \left( \frac{1}{V} \int_M \varphi(t) \Omega \right) = \frac{1}{V} \int_M \dot{\varphi}\, \Omega = \frac{1}{V} \int_M \left( \log \frac{\omega^n}{\Omega} \right) \Omega \le \log \left( \frac{1}{V} \int_M \omega^n \right) =0,$$
by Jensen's inequality.  Hence
$$\int_M  \varphi(t) \Omega = \inf_{[0,t]} \int_M \varphi \, \Omega$$
and so we have a uniform upper bound for $\tr{\omega_0}{\omega}$.  The result then follows by applying Exercise \ref{exeq} and the volume bound of Lemma \ref{phiagain2}.
\end{proof}

\begin{exer} \label{ex47}
Prove the claim (\ref{claimA}).
\end{exer}

As in the previous section, we have estimates on $\varphi$ to all orders:

\begin{prop} \label{pestimates3} Let $\varphi=\varphi(t)$ solve (\ref{may}) for $t\in [0,\infty)$.  Then for each $k=0,1,2,\ldots$ there exists a positive constant $A_k$ such that on $[0,\infty)$,
$$\| \varphi \|_{C^k(M)} \le A_k, \quad \omega_0 + \ddbar \varphi \ge \frac{1}{A_0} \omega_0.$$
\end{prop}

Now we have estimates for the solution to the K\"ahler-Ricci flow, we know we have sequential $C^{\infty}$ convergence of $\varphi(t)$ to some smooth function $\varphi_{\infty}$, say (not \emph{a priori} unique).  To obtain smooth convergence to a K\"ahler-Einstein metric, we need a further argument.  In the absence of a decay estimate like that of Lemma \ref{c1lem}, part (ii), we use an argument of Phong-Sturm \cite{PS} and 
  make use of a functional that decreases along the flow. 

Define
$$P(t) = \int_M \dot{\varphi} \, \omega^n.$$
We leave it as an exercise to show that
\begin{equation} \label{ddtP}
\frac{d}{dt} P(t) = - \frac{1}{n} \int_M | \partial \dot{\varphi} |_g^2 \omega^n \le 0.
\end{equation}

\begin{exer} \label{ex48}
Prove (\ref{ddtP}).
\end{exer}

From Lemma \ref{phiagain2}, $P(t)$ is bounded.  Since it is decreasing, it follows that for a sequence of times $t_i \in [i, i+1]$, we have
$$\left( \int_M | \partial \dot{\varphi} |_g^2 \omega^n \right)(t_i) = \left( \int_M \left| \partial \log \frac{\omega^n}{\Omega} \right|_g^2 \omega^n \right)(t_i) \rightarrow 0.$$
Indeed if not there would exist $\ve>0$ and infinitely many time intervals $[i_j, i_j+1]$ on which
$$\int_M | \partial \dot{\varphi}|_g^2 \omega^n \ge \ve,$$
contradicting the fact that  
$$P(s) - P(0) = -\frac{1}{n} \int_0^s  \int_M  |\partial \dot{\varphi}|^2_g \omega^ndt$$
is bounded as $s\rightarrow \infty$.

Hence, from the $C^{\infty}$ estimates, $\varphi(t_i)$ converges (after passing to a subsequence) to a smooth function $\varphi_{\infty}$ with 
$$\log \frac{\omega_{\infty}^n}{\Omega} = \textrm{constant}, \quad \textrm{for } \omega_{\infty}:= \omega_0 + \ddbar \varphi_{\infty},$$
and taking $\ddbar$ gives
$$\Ric(\omega_{\infty}) =  \ddbar \log \Omega =0.$$
Then $\omega_{\infty} =\oke$, the unique K\"ahler-Einstein metric in the class $[\omega_0]$.  We have proved smooth convergence of the flow to $\oke$ 
 for some sequence of times $t_i \rightarrow \infty$.  

To prove full convergence, we make use of the following exercise.

\begin{exer} \label{ex49}
Show that 
$$\frac{d^2 P}{dt^2} \ge  C\frac{dP}{dt},$$
for some uniform $C$ (making use of the $C^{\infty}$ estimates).  Hence show that 
$\displaystyle{\frac{dP}{dt} \rightarrow 0}.$
\end{exer}

Given this, we obtain smooth convergence using the uniqueness of solutions to the K\"ahler-Einstein equation (\ref{krflat}).  Indeed, suppose for a contradiction that we do not have convergence of $\omega(t)$ to $\oke$.  Then there exists a sequence of times $t_i \rightarrow \infty$ so that, after passing to a subsequence, $\omega(t_i)$ converges in $C^{\infty}$ to $\omega'_{\infty} \neq \oke$.  But since 
$$\frac{dP}{dt} \rightarrow 0,$$
it follows from the above argument that $\omega'_{\infty} \in [\omega_0]$ also satisfies
$$\Ric(\omega'_{\infty}) =0,$$
contradicting the uniqueness of K\"ahler-Einstein metrics in $[\omega_0]$.  This completes the proof of Theorem \ref{thmcao2}.

\chapter{The K\"ahler-Ricci flow on K\"ahler surfaces, and beyond}

In this lecture we will discuss, informally and without proofs, the behavior of the K\"ahler-Ricci flow on  K\"ahler manifolds of complex dimension two.  We will also describe a flow, known as the Chern-Ricci flow, which makes sense on complex manifolds which do not admit K\"ahler metrics.

\section{Riemann surfaces}

First, let $(M, \omega_0)$ be a compact K\"ahler manifold of complex dimension 1.  All such manifolds either have $c_1(M)<0$, $c_1(M)=0$ or $c_1(M)>0$.  These correspond topologically to surfaces with genus $>1$, genus $1$ (a torus) or genus $0$ (the $2$-sphere).  

We know from the previous lecture that if $c_1(M)<0$  the K\"ahler-Ricci flow exists for all time with the volume of the manifold tending to infinity.  If we rescale the metric so that the volume remains bounded, then the normalized K\"ahler-Ricci flow converges at infinity to a K\"ahler-Einstein metric with negative Ricci curvature.

In the case $c_1(M)=0$  the K\"ahler-Ricci flow exists for all time and converges at infinity to a K\"ahler-Einstein metric with zero Ricci curvature.

This only leaves  $c_1(M)>0$  which is precisely the case of $\mathbb{P}^1$.   
We saw from Example \ref{P1krf} that when starting from the standard Fubini-Study metric, the $\mathbb{P}^1$ shrinks to a point along the K\"ahler-Ricci flow.  It is a deep result of Hamilton  \cite{Hs} and Chow \cite{Ch0} that starting at \emph{any} K\"ahler metric on $\mathbb{P}^1$, the K\"ahler-Ricci flow shrinks to a point in finite time, and if rescaled so that the flow exists for all time, converges smoothly to a K\"ahler-Einstein metric on $\mathbb{P}^1$.  The limiting metric is not necessarily the Fubini-Study metric, but is related to it by a biholomorphism.

This is essentially the full picture for the K\"ahler-Ricci flow on a compact Riemann surface (at least, for smooth initial metrics, cf. \cite{GT, MRS}).

\section{K\"ahler surfaces,  blowing up and Kodaira dimension}

Let $M$ be a compact manifold of complex dimension 2.  We call this a \emph{complex surface} (not to be confused with a Riemann surface!) and a \emph{K\"ahler surface} if it admits a K\"ahler metric, $\omega_0$, say.  There is a classification for complex surfaces, known as the Kodaira-Enriques classification (see \cite{BHPV} for example).  However, it is much more complicated than the picture for Riemann surfaces, and in fact there are still some gaps  to be filled.  

One reason that there are ``many more'' K\"ahler surfaces than Riemann surfaces comes from the \emph{blow-up procedure}.  This is a way of constructing a new complex surface from an old one.   We explain this in the simple case of the (non compact) complex manifold $\mathbb{C}^2$.  The \emph{blow-up of $\mathbb{C}^2$ at $0$} is defined to be
$$\textrm{Bl}_0\mathbb{C}^2 = \{ (z, \ell) \in \mathbb{C}^2 \times \mathbb{P}^1 \ | \ z \in \ell \}.$$
Recall that points $\ell$ in $\mathbb{P}^1$ can be regarded as complex lines through the origin in $\mathbb{C}^2$, so this definition makes sense.  $\textrm{Bl}_0\mathbb{C}^2$ is a complex submanifold of $\mathbb{C}^2 \times \mathbb{P}^1$ of codimension 1 (it is given by a single defining equation) and hence is complex manifold of dimension 2.
There is a holomorphic map 
$$\pi: \textrm{Bl}_0\mathbb{C}^2 \rightarrow \mathbb{C}^2, \qquad \pi(z, \ell) = z,$$
called the \emph{blow-up map}.  This map is certainly surjective, and since each non zero element of $\mathbb{C}^2$ lies on a unique line through the origin in $\mathbb{C}^2$, $\pi$ is injective away from $\pi^{-1}(0)$.  Since zero lies in every $\ell$ in $\mathbb{P}^1$, we have $\pi^{-1}(0) \cong \mathbb{P}^1$.  This set is called the \emph{exceptional curve}, which we write as $E$.  
The map $\pi$ is in fact a biholomorphism from $\textrm{Bl}_0\mathbb{C}^2 - E$ to $\mathbb{C}^2 - \{ 0 \}$, and maps $E$ to $0$ (see for example \cite{GH}).

What we have done here is replace a single point $0$ in $\mathbb{C}^2$ with a copy of $\mathbb{P}^1$, which we call the exceptional curve $E$, which represents all of the directions through $0$.  This is in fact a local process, and can be performed on any complex surface $M$ with a designated point $p$ to produce a new complex surface $\textrm{Bl}_pM$ of the same dimension, the \emph{blow-up of $M$ at $p$}.  The new surface $\textrm{Bl}_pM$ has ``more  topology'' than $M$ due to this extra $\mathbb{P}^1$, and indeed the second Betti number of $\textrm{Bl}_pM$ is exactly one more than the second Betti number of $M$.

We can reverse the process of blowing up.  If a complex surface $M$ contains a $\mathbb{P}^1$ which looks locally like the  $\mathbb{P}^1$ inside $\textrm{Bl}_0\mathbb{C}^2$ (i.e. same normal bundle) then we say this is an \emph{exceptional curve} $E$.  It's a theorem that there exists a map $\pi :M \rightarrow N$ to a new surface $N$ which \emph{blows down} the curve $E\subset M$ to a point $p\in N$.

If $M$ contains no exceptional curves, then we say that $M$ is \emph{minimal}.  $M$ can also be called a \emph{minimal model}.  Given the results we just stated, it is rather easy to see that given any compact complex surface $M$ we can obtain a minimal model by a finite sequence of blow downs.  Indeed, if an exceptional curve exists then blow it down.  This process reduces the second Betti number by one and hence must terminate after finitely many steps.  This simple algorithm is the baby version of the \emph{minimal model program} and was known to the classical algebraic geometers.  Its analogue in higher dimensions is far more complicated and the subject of much recent research (see for example \cite{BCHM}).

We now return to the K\"ahler-Ricci flow.  A calculation shows that if we integrate the Ricci curvature of a K\"ahler metric over an exceptional curve $E$, we obtain
$$\int_E \textrm{Ric}(\omega) = 2\pi,$$
and this formula is independent of the choice of K\"ahler metric and  of the surface in which the exceptional curve $E$ is contained (algebraic geometers write this formula as $K\cdot E=-1$). Along the K\"ahler-Ricci flow we have \cite{FIK}
$$\frac{d}{dt} \int_E \omega = - \int_E \textrm{Ric}(\omega) = - 2\pi.$$
So exceptional curves shrink along the K\"ahler-Ricci flow.  Feldman-Ilmanen-Knopf asked \cite{FIK}:  does the K\"ahler-Ricci flow blow down exceptional curves?  

Before answering this question, we make a brief digression to define the important concept of \emph{Kodaira dimension}.  Let $M$ be a 
 compact complex manifold of dimension $n$.  Write $K$ for the \emph{canonical bundle of $M$}, namely the line bundle of holomorphic $(n,0)$ forms on $M$.  Write $H^0(M, K)$ for the vector space of global holomorphic sections of $K$.  Namely, $H^0(M,K)$ is the vector space of holomorphic $(n,0)$ forms on $M$ (which could be the set $\{ 0 \}$).   Then for $\ell =1, 2, \ldots$, the space $H^0(M, K^{\ell})$ is the space of global holomorphic sections of $K^{\ell}$ (the $\ell$th tensor power).  If a line  bundle has many holomorphic sections, then its tensor powers will have many more.  The Kodaira dimension measures the growth of the dimension of $H^0(M, K^{\ell})$ as $\ell \rightarrow \infty$.
 
Define the \emph{Kodaira dimension of $M$} to be the smallest integer $\textrm{Kod}(M)$ such that 
 $$\dim H^0(M, K^{\ell}) \le C \ell^{\textrm{Kod}(M)}, \quad \textrm{ for $\ell$ large},$$
 with the convention that if $H^0(M, K^{\ell}) = \{0 \}$ then we take $\textrm{Kod}(M) = -\infty$.  It's a fact from algebraic geometry that $\ell$ must take one of the values $-\infty, 0,1,2, \ldots, n$.
 
We will quote here some basic facts about Kodaira dimension.  The first is that on a Riemann surface $M$ we have:
\begin{itemize}
\item If $c_1(M)<0$ then $\textrm{Kod}(M)=1$.
\item If $c_1(M)=0$ then $\textrm{Kod}(M)=0$.
\item If $c_1(M)>0$ then $\textrm{Kod}(M)=-\infty$.
\end{itemize}
 Indeed this follows from the fact that $c_1(M)<0$ corresponds via the Kodaira embedding theorem to the statement that $K$ is \emph{ample}, meaning that $K^{\ell}$ has lots of global sections for $\ell$ large.  The condition $c_1(M)>0$ corresponds to $K^{-1}$ being ample, which implies that $K^{\ell}$ has no nonzero sections for $\ell \ge 1$.  Finally, if a Riemann surface has $c_1(M)=0$ then $K$ is trivial and $\dim H^0(M, K^{\ell})=1$ for all $\ell \ge 1$.
  
The second fact is that Kodaira dimension has the following additive property:
$$\textrm{Kod} (M_1\times M_2) = \textrm{Kod}(M_1) \times \textrm{Kod}(M_2).$$
Now we can quickly compute some examples in complex dimension two:
\begin{enumerate}
\item[(a)] If $M$ is a product of two Riemann surfaces of genus $>1$ then $\textrm{Kod}(M)=2$.
\item[(b)] If $M$ is a product of a torus and a Riemann surface of genus $>1$ then $\textrm{Kod}(M)=1$.
\item[(c)] If $M$ is a product of two tori then $\textrm{Kod}(M)=0$.
\item[(d)] If $M$ is a product of a $\mathbb{P}^1$ with any other Riemann surface then $\textrm{Kod}(M)=-\infty$.
\end{enumerate}
Now returning to the K\"ahler-Ricci flow:  if we put a product K\"ahler-Einstein metric on each of the examples (a)-(d) we notice that the K\"ahler-Ricci flow exists for all time in cases (a)-(c), whereas in (d) we have collapsing of the $\mathbb{P}^1$ in finite time.  Morally speaking:  the condition $\textrm{Kod}(M)=-\infty$ means that there is some ``positive curvature'' direction which "wants to shrink" along the flow, whereas $\textrm{Kod}(M) \ge 0$ means we have only ``zero curvature'' or ``negative curvature'' directions.

Finally, the third fact is that Kodaira dimension is invariant under blow-ups:
 $$\textrm{Kod}(\textrm{Bl}_p(M)) = \textrm{Kod}(M).$$
  
 \section{Behavior of the K\"ahler-Ricci flow on K\"ahler surfaces} 
  
We now describe the behavior of the K\"ahler-Ricci flow on a K\"ahler surface. We break this up into different cases.

\subsection{Non-minimal K\"ahler surfaces with $\Kod(M)\neq -\infty$.}

 We first consider the case when $\textrm{Kod}(M)\ge 0$ (the case $\Kod(M) = -\infty$ is more complicated and will be discussed later).  We suppose that $M$ is not minimal - i.e. it has at least one exceptional curve.  A
  result of Song and the author \cite{SW2, SW3, SW4} says, roughly speaking, that the K\"ahler-Ricci flow blows down exceptional curves finitely many times until obtaining a minimal surface.
  We state the result somewhat informally:

\begin{thm}  \label{sw} Suppose that $M$ is a compact K\"ahler surface with $\emph{Kod}(M)\neq -\infty$ and assume that $M$ contains at least one exceptional curve.  Then  there exist finitely many disjoint exceptional curves $E_1, \ldots, E_k$ on $M$ and a map $\pi: M \rightarrow M_1$ blowing them down.  The K\"ahler-Ricci flow exists on $[0,T)$ for some $T$ with $0<T <\infty$ and ``blows down'' $E_1, \ldots, E_k$ and continues on the new manifold $M_1$.  This process repeats finitely many times until we obtain $M_{\ell}$ minimal.  On $M_{\ell}$ the K\"ahler-Ricci flow exists for all time. 
\end{thm}
  
It should be explained what is meant by the K\"ahler-Ricci flow ``blowing down'' exceptional curves $E_1, \ldots, E_k$, since this is the essential content of the result (the fact that the flow exists only for a finite time follows easily from Theorem \ref{thmmaximal}, as does the fact that the K\"ahler-Ricci flow exists   for all time on $M_{\ell}$.)

We say that the K\"ahler-Ricci flow \emph{blows down} $E_1, \ldots, E_k$ if, first,  the flow $g(t)$ converges smoothly on compact subsets of $M \setminus \cup E_i$ to a smooth K\"ahler metric $g_T$ and if $(M, g(t))$ converges globally as a metric space to the metric completion of $(M \setminus \cup E_i, g_T)$.  Second, we insist that there exists a smooth solution to the K\"ahler-Ricci flow on $M_1$ for $t>T$ which converges as $t\rightarrow T^+$ to $g_T$ smoothly on compact subsets away from the points $p_i:=\pi(E_i)$.  Third, we require that $(M_1, g(t))$ converges globally as a metric space to $(M \setminus \cup E_i, g_T)$ as $t \rightarrow T^+$.  Here ``converges as a metric space'' means convergence in the sense of Gromov-Hausdorff (we omit the precise definition here).

The fact that the K\"ahler-Ricci flow can be restarted on the new manifold $M_1$ makes use of theorem of Song-Tian \cite{SoT3}.  The study of the K\"ahler-Ricci flow in relation to the minimal model program was initiated by Song and Tian \cite{SoT1, T2, SoT2, SoT3}  and is known as the \emph{analytic minimal model program} (see also \cite{LT}).

\subsection{Minimal surfaces with $\Kod(M)\neq -\infty$} \label{sectionmin}

Next we discuss the case of what happens on a minimal surface.  As stated in the theorem above, the flow exists for all time.  Indeed, from some basic algebraic geometry, every minimal $M$ with $\Kod(M)\neq -\infty$  has  $-c_1(M)$ nef and  we can apply Exercise \ref{ex27} and Theorem \ref{thmmaximal}.  The behavior of the flow as $t \rightarrow \infty$ depends crucially on the Kodaira dimension.

First suppose that $\Kod(M)=2$.  If $c_1(M)<0$ then, from the results discussed in Lecture 4,  the flow converges after normalization to a K\"ahler-Einstein metric.  Otherwise,  the canonical bundle is ``big and nef'' which means that $-c_1(M)$ satisfies a weaker positivity condition.  A result of Tsuji \cite{Ts1} and Tian-Zhang \cite{TZha} shows that the normalized flow converges to a K\"ahler-Einstein metric smoothly on compact subsets of $M \setminus V$ where $V$ is a certain subvariety on $M$.

If $\Kod(M)=1$ then  the manifold is a ``properly elliptic surface''.  Namely, there exists a surjective holomorphic map $f: M \rightarrow S$ to a Riemann surface $S$ with the property that $f^{-1}(s)$ is a torus for all but finitely many $s \in S$.  It was shown by Song-Tian \cite{SoT1} that, in a weak sense, the K\"ahler-Ricci flow collapses these torus fibers and converges to a ``generalized K\"ahler-Einstein metric'' on $S$.  This is a metric whose Ricci curvature is given by the negative of the metric plus some additional terms arising from the non-product structure of $M$.  In the simpler case when $M$ is a product or a smooth fibration (with $S$ necessarily a Riemann surface of genus $>1$), it was shown in (\cite{SW4, Gi2, FZ}, see also \cite{GTZ}) that the flow converges smoothly to the K\"ahler-Einstein metric on $S$ as $t\rightarrow \infty$.

Finally if $\Kod(M)=0$ then $c_1(M)=0$ and, by the result discussed in Lecture 4, the K\"ahler-Ricci flow converges smoothly to a K\"ahler-Einstein metric with zero Ricci curvature.

Combining these results with that of Theorem \ref{sw} we see that the K\"ahler-Ricci flow is largely understood if $\Kod(M)\neq -\infty$.  Moreover,  the behavior of the flow reflects the geometry of the underlying complex manifold.

\subsection{The case of $\Kod(M)=-\infty$} \label{sectinfty}

We now discuss  the more troublesome case when $\Kod(M)=-\infty$.  Indeed even in the simple case of $\mathbb{P}^1 \times \mathbb{P}^1$, Exercise \ref{P1P1e} shows that very different behavior can occur if different initial metrics are chosen.  

We will focus on a slightly more complicated example:  let $M$ be $\mathbb{P}^2$ blown up at a single point $p$.  Recalling the definition of the blow up of $\mathbb{C}^2$ at the origin, we see that in addition to the map $\pi: \textrm{Bl}_0\mathbb{C}^2 \rightarrow \mathbb{C}^2$ there is another map
$$f:  \textrm{Bl}_0\mathbb{C}^2 = \{ (z, \ell) \in \mathbb{C}^2 \times \mathbb{P}^1 \ | \ z \in \ell \} \rightarrow \mathbb{P}^1,$$
given by projection $f(z, \ell) = \ell$ onto the second factor.  If we compactify $\mathbb{C}^2$ to $\mathbb{P}^2$ we have maps
\[ 
\begin{split}
\textrm{Bl}_p\mathbb{P}^2 {} &  \overset{\pi}{\longrightarrow}   \mathbb{P}^2 \\
\downarrow f \ \ & \\
\mathbb{P}^1\ \ \  & 
\end{split}
\]
and $f$ is a bundle map whose fibers are isomorphic to $\mathbb{P}^1$.  Write $\omega_{\mathbb{P}^2}$ and $\omega_{\mathbb{P}^1}$ for the Fubini-Study metric on $\mathbb{P}^2$ and $\mathbb{P}^1$ respectively.

The K\"ahler cone of $M$ is given by
$$\textrm{Ka}(M) = \{  x [f^*\omega_{\mathbb{P}^1}] + y [\pi^* \omega_{\mathbb{P}^2}]  \ | \ x, y \in \mathbb{R}^{>0} \},$$
and the first Chern class of $M$ by
$$c_1(M) = 2[\pi^*\omega_{\mathbb{P}^2}] + [f^*\omega_{\mathbb{P}^1}]>0.$$
There are three different possible behaviors of the K\"ahler-Ricci flow, depending on where the initial K\"ahler class $[\omega_0]$ lies, as illustrated by Figure \ref{fp2bu}.

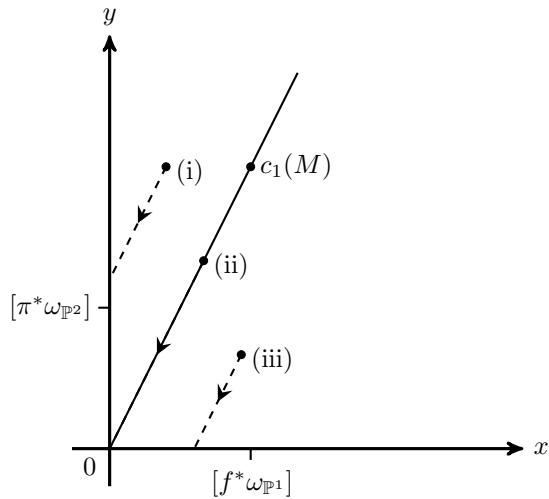
\begin{figure}[h!]
\begin{center}
\begin{tikzpicture}[scale=5, important line/.style={thick}, ax
    dashed line/.style={dashed,->, thick}, axis/.style={very thick, ->, >=stealth'},
        every node/.style={color=black} ]
        \draw[axis] (-0.1,0)  -- (1.1,0) node(xline)[right]
        {$x$};
    \draw[axis] (0,-0.1) -- (0,1.1) node(yline)[above] {$y$};
    \draw[important line] (0,0) -- (.5,1);
    \draw[directed, dashed, thick] (.15, .75) node[right, text width=1em] {\ (i)}  -- (0,.45);
    \draw[directed, dashed, thick] (.25, .5) node[right, text width=1em] {\ (ii)} -- (0,.0);
    \draw[directed, dashed, thick] (.35, .25)  node[right, text width=1em] {\ (iii)} -- (0.225,.0);
    \draw[fill] (.15, .75) circle (.3pt);
    \draw[fill] (.25,.5) circle (.3pt);
    \draw[fill] (.35, .25) circle (.3pt);
    \draw[fill] (.375,.75) circle (.3pt) node[right, text width=2em] {$c_1(M)$};
    \draw (0,0) node[left, below, text width=2em] {$0$};
    \draw[important line] (0.375,0) -- (0.375,-.03) node[below, text width=3em] {$[f^*\omega_{\mathbb{P}^1}]$};
    \draw[important line] (0,0.375) -- (-0.03,.375) node[left, text width=3em] {$[\pi^*\omega_{\mathbb{P}^2}]$};
 \end{tikzpicture}    
\end{center}
\caption{Three behaviors of the flow on the blow up of $\mathbb{P}^2$} \label{fp2bu}
\end{figure}

\begin{enumerate}
\item[(i)]  If $[\omega_0]$ lies above the line $y=2x$ then the K\"ahler-Ricci flow blows down the exceptional curve in the sense described above, and continues on $\mathbb{P}^2$ \cite{SW1, SW2, SW3}.
\item[(ii)]  If $[\omega_0]$ lies on the line $y=2x$ containing $c_1(M)$ then the K\"ahler-Ricci flow shrinks to a point in finite time \cite{P4, SeT}.  After rescaling and reparametrizing converges to a K\"ahler-Ricci soliton (which is a solution of the K\"ahler-Ricci flow which moves by automorphism) \cite{Cao3, Ko, Zhu, TZhu,  PSSW3, TZZZ, CS}. 
\item[(iii)] If $[\omega_0]$ lies below the line $y=2x$ then the K\"ahler-Ricci flow contracts the $\mathbb{P}^1$ fibers and converges at least by sequence in the sense of metric spaces to a metric on the base $\mathbb{P}^1$ \cite{SW1, SSW}.
\end{enumerate}
A general K\"ahler surface with $\textrm{Kod}(M)=-\infty$ is either $\mathbb{P}^2$ or a $\mathbb{P}^1$ bundle over a Riemann surface or is obtained by blowing up one of these manifolds.    The K\"ahler-Ricci flow always exhibits one of the three behaviors (i), (ii) or (iii) described above.  In (ii),  the K\"ahler-Ricci soliton may be ``trivial'' - i.e. a K\"ahler-Einstein metric \cite{TY, T, WZ}.

\subsection{Some open problems}

Although much is understood about the K\"aher-Ricci flow in the case of K\"ahler surfaces, and the general conjectural picture is now well-laid out, there are a number of difficult problems that remain, and we mention here just a few.

A well-known problem is to understand more precisely the singularity formation when an exceptional curve contracts.  Feldman-Ilmanen-Knopf \cite{FIK} conjectured that the blow-up limit, obtaining by rescaling the metric around the singular time,  should yield the shrinking non-compact K\"ahler-Ricci soliton that they constructed.  This conjecture was confirmed in the example of Section \ref{sectinfty} assuming symmetric data by M\'aximo \cite{Ma}, using a result of Song \cite{So2} that in the symmetric case the singularity is of Type I (meaning that the curvature bound $|\textrm{Rm}|\le C/(T-t)$ holds).  Related to this is a folklore conjecture that all finite time singularities of the K\"ahler-Ricci flow are of Type I.  Even the question of whether the bound on the scalar curvature $R\le C/(T-t)$ holds is open (cf. \cite{Zhang}).

The result of Theorem \ref{sw} for the K\"ahler-Ricci flow makes use of results from algebraic geometry and K\"ahler surface theory to prove existence of exceptional curves.  It would be a long term goal to use the flow to construct algebraic objects such as exceptional curves and give new analytic proofs of results in algebraic geometry.  In this direction,  it was shown by Collins-Tosatti \cite{CT} that in general whenever the K\"ahler-Ricci flow encounters a noncollapsing singularity, the metrics develop singularities precisely along an algebraic variety (proving a conjecture of \cite{FIK}).   

Another problem is to understand the global metric behavior of the K\"ahler-Ricci flow as $t \rightarrow \infty$ in the case of a minimal surface with $\Kod(M)=2$ when $c_1(M)$ is not negative.  In the case when the variety $V$ (as discussed above) is a union of disjoint curves of self-intersection $-2$, it was shown in \cite{SW3} that the K\"ahler-Ricci flow converges  in the sense of metric spaces to an orbifold K\"ahler-Einstein manifold.  The general case of  possibly  intersecting curves is still open.  

It is also an open problem to understand precisely the behavior of the K\"ahler-Ricci flow as $t \rightarrow \infty$ on a minimal surface with $\Kod(M)=1$ which is not a smooth fibration.  The results of Song-Tian \cite{SoT1} give information about convergence of the flow in the sense of currents, but it is still unknown exactly what happens to the metrics.  Here, difficulties arise from the presence of singular and multiple fibers.

A difficult problem is to understand collapsing along the K\"ahler-Ricci flow in the case of negative Kodaira dimension.  Surprisingly, it is still an open problem to determine the precise behavior of the flow even in the case of $\mathbb{P}^1 \times \mathbb{P}^1$ when one of the fibers collapses.  It was shown by Song-Sz\'ekelyhidi-Weinkove \cite{SSW} that the diameter of the collapsing fiber is bounded above by a multiple of $(T-t)^{1/3}$, but this falls short of the optimal rate of $(T-t)^{1/2}$ (see also \cite{Fo2}).  It is expected that the blow-up limit is a product with a flat direction and this has been proved under symmetry conditions \cite{Fo, So2}.
Underlying this difficulty is the depth of the problem of understanding the K\"ahler-Ricci flow for a general initial metric on $\mathbb{P}^1$ (the result of Hamilton and Chow), which still has no simple proof.  In higher dimensions the problem is far harder, since one could replace the fiber $\mathbb{P}^1$ by a general Fano manifold, where the behavior of the flow is far more difficult to understand (and this goes way beyond the scope of these notes).

Finally, of course, one would like to extend all of these ideas to higher dimensions.  Considerable progress is being made  \cite{SoT2, SoT3, So3, SY, GZ} and this will continue to be a challenging and exciting area of research for many years to come.

\section{Non-K\"ahler surfaces and the Chern-Ricci flow}

As seen from the last section, the K\"ahler-Ricci flow is now quite well understood in the case of complex dimension two.   Given any K\"ahler surface, we have a more-or-less complete picture of how the flow will behave (modulo some open problems, as discussed above).  This is in contrast to the Ricci flow on general four-manifolds, where despite the success of Ricci flow in three dimensions \cite{H1,H3,P1,KL, CZ,MT1} we do not yet have any kind of conjectural picture.  Is there a larger class of four-manifolds than K\"ahler surfaces for which we can say anything?

We consider the class of compact complex surfaces, which include \emph{non-K\"ahler} surfaces, namely surfaces which do not admit any K\"ahler metric.  For an  example, consider the simplest Hopf surface
$$H = (\mathbb{C}^2- \{ 0 \}) / \sim,$$
where $(z^1, z^2) \sim (2z^1, 2z^2)$.  This is a compact complex surface, diffeomorphic to $S^3 \times S^1$ via the map
$$z \mapsto \left( \frac{z}{|z|}, |z| \right) \in S^3 \times \mathbb{R}^{>0}/(r \sim 2r) \cong S^3 \times S^1,$$ 
where we consider $S^3$ as a subset of $\mathbb{C}^2$ in the usual way.  $H$  cannot admit a K\"ahler metric since its second Betti number vanishes.
  
All complex surfaces admit Hermitian metrics.  Given  such a metric $g_0$, we can consider the associated $(1,1)$ form
$$\omega_0 = \sqrt{-1} g_{i\ov{j}} dz^i \wedge d\ov{z}^j,$$
which is not necessarily closed.  We define the \emph{Chern-Ricci form} of $\omega_0$ to be
$$\textrm{Ric}(\omega_0) = -\ddbar \log \det g_0.$$
This coincides with the usual Ricci curvature if $g_0$ is K\"ahler.  The form $\textrm{Ric}(\omega_0)$ is still a closed $(1,1)$-form for Hermitian $g_0$, but the key point is that in the non-K\"ahler case $\textrm{Ric}(\omega_0)$ is not in general equal to the Riemannian Ricci curvature of $g_0$.

 We consider a parabolic flow of Hermitian metrics on a complex surface $M$:
\begin{equation} \label{crf}
\ddt{} \omega = - \textrm{Ric}(\omega), \quad \omega|_{t=0} = \omega_0.
\end{equation}
The formula is the same as for the K\"ahler-Ricci flow, but we are allowing $g_0$ to be non-K\"ahler.   The equation (\ref{crf}) is known as the \emph{Chern-Ricci flow}, and shares many of the properties of the K\"ahler-Ricci flow \cite{Gi, Gi3, TW2, TW3, TWY, ShW2}.

Note that the Chern-Ricci flow is \emph{not} the same as the Ricci flow in general.  The Ricci flow starting at a Hermitian metric may immediately become non-Hermitian and little is known about the behavior of the flow.
Other examples of  Hermitian flows generalizing the K\"ahler-Ricci flow have been given in \cite{ST}.

The Chern-Ricci flow  was first introduced by M. Gill \cite{Gi} in the setting of manifolds with vanishing first Bott-Chern class, which we now explain.  We define
$$H^{1,1}_{\textrm{BC}}(M, \mathbb{R}) = \frac{ \{ \textrm{$\ov{\partial}$-closed real $(1,1)$ forms} \}}{\textrm{Im} \, \partial \ov{\partial}},$$
which coincides with $H^{1,1}_{\ov{\partial}}(M, \mathbb{R})$ when $M$ is K\"ahler, as discussed in Lecture 2.  We then define the first Bott-Chern class of $M$ by
$$c_1^{\textrm{BC}}(M) = [\textrm{Ric}(\omega_0)] \in H^{1,1}_{\textrm{BC}}(M, \mathbb{R}),$$
and by the same argument as in Lecture 2, this is independent of choice of Hermitian metric $\omega_0$.

Gill \cite{Gi} proved:

\begin{thm} \label{thmgill}
If $M$ is a compact complex manifold with $c_1^{\textrm{BC}}(M)=0$ then there exists a unique  solution $\omega(t)$ to the Chern-Ricci flow (\ref{crf}) starting at any Hermitian metric $\omega_0$.  As $t \rightarrow \infty$, 
$$\omega(t) \rightarrow \omega_{\infty} \  \textrm{in } C^{\infty}(M),$$
where $\emph{Ric}(\omega_{\infty})=0$.
\end{thm}

This is the analogue of Cao's Theorem \ref{thmcao2} proved in Lecture 4.  This result made use of an $L^{\infty}$ estimate for the complex Monge-Amp\`ere equation in the Hermitian case due to Tosatti and the author \cite{TW1} (see also \cite{Che, GL}).  Note that in fact this result holds in all dimensions.  

An analogue of Theorem \ref{thmcao} also exists.  As there, it is convenient to consider the normalized flow:
\begin{equation} \label{ncrf}
\ddt{} \omega = - \textrm{Ric}(\omega) - \omega, \quad \omega|_{t=0} = \omega_0.
\end{equation}

  It was shown by Tosatti and the author \cite{TW2} that:

\begin{thm} \label{tw1}
Let $M$ be a compact complex manifold with $c_1(M)<0$.  Then there exists a unique solution to the  Chern-Ricci flow (\ref{crf}) starting at any Hermitian metric $\omega_0$.  As $t \rightarrow \infty$, the solution $\omega(t)$ to the normalized flow (\ref{ncrf}) satisfies
$$\omega(t) \rightarrow \omega_{\emph{KE}} \ \textrm{in } C^{\infty}(M),$$
where $\omega_{\emph{KE}}$ is the unique K\"ahler-Einstein metric on $M$ satisfying $\emph{Ric}(\omega_{\emph{KE}}) = - \omega_{\emph{KE}}.$
\end{thm}

Observe that the condition $c_1(M)<0$ implies that the manifold $M$ is K\"ahler.  The point of this theorem is that the Chern-Ricci flow takes any \emph{non-K\"ahler} Hermitian metric to the K\"ahler-Einstein metric.

Furthermore, we have a natural analogue of the maximal existence time theorem \cite{TW2}:

\begin{thm}  Given any Hermitian metric $\omega_0$, there exists a unique maximal solution of the Chern-Ricci flow starting at $\omega_0$ on $[0,T)$, where
\[
T= \sup \left\{ t>0 \ \bigg| \ \begin{array}{l} \emph{there exists $\psi \in C^{\infty}(M)$ such that  } \\
\ \ \omega_0 - t \emph{Ric}(\omega_0) + \ddbar \psi >0 \end{array} \right\}. 
\]
\end{thm}
Although it looks like $T$ depends only on $\omega_0$, it really only depends on the ``equivalence class'' of $\omega_0$, where we say that two Hermitian metrics are equivalent if their forms differ by the $\partial \ov{\partial}$ of a function.  In many cases, it is easy to compute $T$, just as in the K\"ahler case.

We return now to the case of complex dimension $2$ and impose an additional assumption on $\omega_0$:
\begin{equation} \label{ga}
\ddbar \omega_0 =0.
\end{equation}
This is a natural assumption on complex surfaces, since, by a theorem of Gauduchon \cite{G},  given any Hermitian metric $\omega$ on $M$ there exists a smooth function $\sigma$ so that $e^{\sigma} \omega$ satisfies (\ref{ga}).  A metric on a complex surface which satisfies (\ref{ga}) is called \emph{Gauduchon} (in dimension  $n$, the condition is $\ddbar \omega^{n-1}=0$.)  It is immediate from the definition that the Chern-Ricci flow preserves the condition (\ref{ga}).

If $M$ is a \emph{minimal} complex surface, then a similar picture as in the K\"ahler case (Section \ref{sectionmin})  is now emerging: 
\begin{itemize}
\item If $\textrm{Kod}(M)=0$ then $c_1^{\textrm{BC}}(M)=0$ and Gill's Theorem \ref{thmgill} implies that the Chern-Ricci flow exists for all time and converges to a Chern-Ricci flat metric. 
\item  If $\textrm{Kod}(M)=1$ and $M$ is non-K\"ahler, then $M$ is, up to a finite covering, a smooth elliptic bundle over a Riemann surface $S$.  A result of Tosatti-Weinkove-Yang \cite{TWY} says that the normalized Chern-Ricci flow (\ref{ncrf}) exists for all time and converges in the sense of metric spaces to an orbifold K\"ahler-Einstein metric on $S$.  
\item If $\textrm{Kod}(M)=2$ then $M$ admits a K\"ahler metric (in fact $M$ is projective algebraic).  If 
  $c_1(M)<0$ and we start the flow from a Hermitian metric,  we can apply Theorem \ref{tw1} above to obtain convergence of the normalized  Chern-Ricci flow to a K\"ahler-Einstein metric.   Otherwise the canonical bundle is big and nef and a result of Gill \cite{Gi3} says that the flow converges smoothly to a K\"ahler-Einstein metric outside a subvariety, generalizing the results of \cite{Ts1, TZha}.
\end{itemize}
  
If $M$ is non-minimal with $\Kod(M)\neq -\infty$, then as in the case of the K\"ahler-Ricci flow we can ask whether the Chern-Ricci flow ``blows down'' exceptional curves.  This is in general an open problem.  However,  it was shown in \cite{TW2} that one obtains smooth convergence outside the curves of negative self-intersection.  Moreover, the curves contract in the sense of metric spaces if the initial metric $\omega_0$ satisfies the additional assumption that $d\omega_0$ is the exterior derivative of the pull-back of a form from the blow-down manifold \cite{TW3}.   It is not difficult to find examples when this condition is satisfied, but we conjecture that one should be able to remove this assumption.

The case of $\Kod(M)=-\infty$ is both more interesting and more difficult.  The  minimal non-K\"ahler surfaces with $\Kod(M)=-\infty$ are known as Class VII surfaces.
When $b_2=0$, $M$ is either a Hopf surface or an Inoue surface.  
It was shown in \cite{TW2} that on a Hopf surface, the Chern-Ricci flow \emph{collapses} in finite time, meaning that the volume tends to zero.  By contrast, on the Inoue surface the flow exists for all time.  Explicit examples were given in \cite{TW3} for a family of Hopf surfaces  which exhibit collapsing in the sense of Gromov-Hausdorff to $S^1$.  Examples on Inoue surfaces also show collapsing of the normalized flow as $t\rightarrow \infty$ \cite{TW3}.  

When $M$ is a Class VII surface with $b_2>0$, the exact behavior of the flow is a mystery, and no explicit examples are known.  However, it was shown in \cite{TW2} that the Chern-Ricci flow always collapses in finite time.   It is a major open problem to complete the classification of Class VII surfaces when $b_2 >1$ (for the cases $b_2=0,1$ see \cite{In, LYZ, Na, T0, T1}) and this is a motivating factor for studying non-K\"ahler surfaces.

\appendix

\chapter{Solutions to exercises}

\noindent
{\bf \ref{S2}.} $S^2= \{ (x^1, x^2, x^3) \in \mathbb{R}^3 \ | \ (x^1)^2+(x^2)^2+(x^3)^2=1 \}$ is a complex manifold with charts $(S^2 - \{ (0,0,1) \}, w)$ and $(S^2 - \{ (0,0,-1) \}, \tilde{w})$ given by
$$w = \frac{x^1+\sqrt{-1}x^2}{1-x^3}, \quad \tilde{w} = \frac{x^1-\sqrt{-1} x^2}{1+x^3},$$
which are related by $w=1/\tilde{w}$ on the overlap.  On the other hand $\mathbb{P}^1$ has two complex charts $U_0 = \{ Z_0 \neq 0\}$ with $z=Z_1/Z_0$ and $U_1 = \{ Z_1 \neq 0 \}$ with $\tilde{z}=Z_0/Z_1$, which are  related by $z=1/\tilde{z}$. All of the maps $w, \tilde{w}, z, \tilde{z}$ map onto $\mathbb{C}$.
Then  define a map $S^2 \rightarrow \mathbb{P}^1$ by mapping $S^2 - \{ 0,0,1\} \rightarrow U_0$ via $ z^{-1}\circ  w$ and $S^2 - \{ 0,0,-1\} \rightarrow U_1$ via $\tilde{z}^{-1} \circ \tilde{w}$.  This is well-defined and holomorphic with holomorphic inverse, hence a diffeomorphism.

\noindent
{\bf \ref{exh}.}   Let $(U,\tilde{z})$ be another coordinate chart.  If $\displaystyle{ \frac{\partial X^i}{\partial \ov{z}^j}=0}$ then on $U \cap \tilde{U}$,
$$ \frac{\partial \tilde{X}^k}{\partial \ov{\tilde{z}}^{\ell}} = \frac{\partial \ov{z}^j}{\partial \ov{\tilde{z}}^{\ell}} \frac{\partial}{\partial \ov{z}^j} \left( X^i \frac{\partial \tilde{z}^k}{\partial z^i} \right)= \frac{\partial \ov{z}^j}{\partial \ov{\tilde{z}}^{\ell}} \frac{\partial X^i}{\partial \ov{z}^j}  \frac{\partial \tilde{z}^k}{\partial z^i}=0.$$

\noindent
{\bf \ref{exinp}.} On $U \cap \tilde{U}$, from (\ref{gtrans}),
$$g_{i\ov{j}} X^i \ov{Y^j} = \tilde{g}_{k\ov{\ell}} \frac{\partial \tilde{z}^k}{\partial z^i} X^i \ov{\frac{\partial \tilde{z}^{\ell}}{\partial z^j} Y^j} = \tilde{g}_{k\ov{\ell}} \tilde{X}^k \ov{\tilde{Y}^j}.$$

\noindent
{\bf \ref{ex14}.} Take the inverse of both sides of (\ref{gtrans}).

\noindent
{\bf \ref{ex15}.} Assume (\ref{kahlercondition}) holds on $U$.  
  On $U \cap \tilde{U}$,
\[
\begin{split}
\frac{\partial}{\partial \tilde{z}^m} \tilde{g}_{p \ov{q}} - \frac{\partial}{\partial \tilde{z}^p} \tilde{g}_{m \ov{q}} = {} & \frac{\partial}{\partial \tilde{z}^m} \left( g_{i\ov{j}}  \frac{\partial z^i}{\partial \tilde{z}^p} \ov{\frac{\partial z^j}{\partial \tilde{z}^q}} \right) -  \frac{\partial}{\partial \tilde{z}^p} \left( g_{i\ov{j}} \frac{\partial z^i}{\partial \tilde{z}^m} \ov{\frac{\partial z^j}{\partial \tilde{z}^q}}  \right) \\
= {} & \left(  \frac{\partial z^k}{\partial \tilde{z}^m}  \frac{\partial}{\partial z^k} g_{i\ov{j}} \right) \frac{\partial z^i}{\partial \tilde{z}^p} \ov{\frac{\partial z^j}{\partial \tilde{z}^q}} + g_{i\ov{j}} \frac{\partial^2 z^i}{\partial \tilde{z}^m \partial \tilde{z}^p} \ov{\frac{\partial z^j}{\partial \tilde{z}^q}} \\
& -  \left(  \frac{\partial z^k}{\partial \tilde{z}^p}  \frac{\partial}{\partial z^k} g_{i\ov{j}} \right) \frac{\partial z^i}{\partial \tilde{z}^m} \ov{\frac{\partial z^j}{\partial \tilde{z}^q}} - g_{i\ov{j}} \frac{\partial^2 z^i}{\partial \tilde{z}^p \partial \tilde{z}^m} \ov{\frac{\partial z^j}{\partial \tilde{z}^q}} \\
 = {} & \left( \frac{\partial}{\partial z^k} g_{i\ov{j}} - \frac{\partial}{\partial z^i} g_{k\ov{j}} \right) \frac{\partial z^k}{\partial \tilde{z}^m} \frac{\partial z^i}{\partial \tilde{z}^p} \ov{\frac{\partial z^j}{\partial \tilde{z}^q}} =0.
\end{split}
\]

\noindent
{\bf \ref{exfs}.} To show that it is well-defined tensor, use the general fact that $\displaystyle{\frac{\partial}{\partial z^i} \frac{\partial}{\partial  \ov{z}^j}}$ acting on functions transforms according to
$$\frac{\partial}{\partial z^i} \frac{\partial}{\partial \ov{z}^j} = \dtu{k}{i} \ov{\dtu{\ell}{j}} \dt{k} \frac{\partial}{\partial \ov{\tilde{z}^{\ell}}}, \quad \textrm{on } U \cap \tilde{U},$$
together with the fact that $\ddbar \log |f|^2=0$ for $f$ a nowhere vanishing holomorphic function.

To see that $(g_{i\ov{j}})$ is positive definite, write $\underline{z}=(z^1, \ldots, z^n)$ and compute on e.g. $U_0$, 
$$(g_{i\ov{j}}) = \left( \frac{\delta_{ij}+|\underline{z}|^2 \delta_{ij} -z^j \ov{z^i}}{(1+ |\underline{z}|^2)^2} \right) \ge  \left( \frac{\delta_{ij}}{(1+ |\underline{z}|^2)^2} \right)>0,$$
since by the Cauchy-Schwarz inequality $|\underline{z}|^2 \delta_{ij} -z^j \ov{z^i}$ is semipositive.

It is immediate from the definition that $\partial_k g_{i\ov{j}} = \partial_i g_{k\ov{j}}$, so $g$ is K\"ahler.

\noindent
{\bf \ref{ex17}.}  $\displaystyle{\partial \omega = \left( \partial_k g_{i\ov{j}} - \partial_i g_{k\ov{j}} \right) dz^k \wedge dz^i \wedge d\ov{z}^j}$ and hence $\partial \omega=0$ if and only if $\omega$ is K\"ahler.  Taking conjugates, $\ov{\partial}\omega = 0$ if and only if $\partial \omega=0$ since $\omega$ is real.  Finally, $d\omega = \partial \omega+ \ov{\partial}{\omega}$ and $\partial \omega$ is of type $(2,1)$ whereas $\ov{\partial} \omega$ is of type $(1,2)$.  Hence $d\omega=0$ if and only if both $\partial \omega$ and $\ov{\partial}\omega$ vanish.

\noindent
{\bf \ref{ex18}.} E.g., for $\nabla_kX^i$, compute
\[
\begin{split}
\du{k} X^i + \Gamma^i_{km} X^m = {} & \dtu{p}{k} \dt{p} \left( \tilde{X}^{\ell} \dut{i}{\ell} \right) + g^{i\ov{q}} \left( \du{k} g_{m\ov{q}} \right) \tilde{X}^{\ell} \dut{m}{\ell} \\
= {} & \dtu{p}{k} \dut{i}{\ell} \frac{\partial \tilde{X}^{\ell}}{\partial \tilde{z}^p} + \dtu{p}{k} \tilde{X}^{\ell} \frac{\partial^2 z^i}{\partial \tilde{z}^p \partial \tilde{z}^{\ell}} \\
& + \tilde{g}^{a\ov{b}} \dut{i}{a} \ov{\dut{q}{b}} \dtu{e}{k} \dt{e} \left( \tilde{g}_{c\ov{d}} \dtu{c}{m} \ov{\dtu{d}{q}} \right) \tilde{X}^{\ell} \dut{m}{\ell}
\end{split}
\]
and use the fact that 
$$\dt{e} \left( \dtu{c}{m}\right) = -\dtu{c}{r} \dtu{s}{m} \frac{\partial^2 z^r}{\partial \tilde{z}^e \partial \tilde{z}^s}.$$

\noindent
{\bf \ref{ex119}.} Straightforward calculation.

\noindent
{\bf \ref{ex110}.} Compute at $0$,
$$\du{k} g_{i\ov{j}} = \du{k} \left( \tilde{g}_{p\ov{q}} \dtu{p}{i} \ov{\dtu{q}{j}} \right) = \dtu{m}{k} \left( \dt{m} \tilde{g}_{p\ov{q}} \right) \dtu{p}{i} \ov{\dtu{q}{j}} + \tilde{g}_{p\ov{q}} \frac{\partial^2 \tilde{z}^p}{\partial z^k \partial z^i} \ov{\dtu{q}{j}}.$$
But at $0$ we have $\displaystyle{\tilde{g}_{p\ov{q}}= \delta_{pq} = \dtu{p}{q}}$ and hence
$$\du{k} g_{i\ov{j}} = \tilde{\Gamma}^j_{ki}(0) - \tilde{\Gamma}^j_{ki}(0)=0.$$

\noindent
{\bf \ref{ex111}.} For example, in holomorphic normal coordinates,
$$[\nabla_i, \nabla_{\ov{j}}] b_{\ov{q}} = \partial_i (\partial_{\ov{j}} b_{\ov{q}} - \ov{\Gamma^{\ell}_{jq}} b_{\ov{\ell}}) - \partial_{\ov{j}} \partial_i b_{\ov{q}} = - (\partial_i \ov{\Gamma^{\ell}_{jq}})b_{\ov{\ell}} = R_{i\ov{j} \ \, \ov{q}}^{\ \ \, \ov{\ell}} b_{\ov{\ell}}.$$

\noindent
{\bf \ref{extr}.} (1) \ Pick coordinates at a point for which $g_{i\ov{j}}=\delta_{ij}$ and $(\beta_{i\ov{j}})$ is a diagonal matrix with eigenvalues $\lambda_1, \ldots, \lambda_n$.  Then $n \omega^{n-1} \wedge \beta$ and $g^{i\ov{j}}\beta_{i\ov{j}} \omega^n$ both equal 
$$\left( \sum_{i=1}^n \lambda_i  \right) n!  (\sqrt{-1})^n dz^1 \wedge d\ov{z}^1 \wedge \cdots \wedge dz^n \wedge d\ov{z}^n.$$
(2) follows from (1).

\noindent
{\bf \ref{exOmega}.}  On $U \cap \tilde{U}$, we have
$$\tilde{a} = \left| \det \left( \dtu{i}{j} \right) \right|^2 a.$$  The exercise follows from the fact that if $f$ is a nowhere vanishing holomorphic function then  $\displaystyle{\ddbar \log |f|^2=0.}$

\noindent
{\bf \ref{exfs2}.} With the notation of Exercise \ref{exfs},
$$\det \left( \frac{\delta_{ij} + |\underline{z}|^2 \delta_{ij} - z^j \ov{z}^i}{(1+ |\underline{z}|^2)^2} \right) = \frac{1}{(1+|\underline{z}|^2)^{n+1}},$$
which can be more easily calculated  by applying a unitary transformation to $\mathbb{C}^n$ so that $z^2=\cdots = z^n=0$.
Then $\Ric(\omega_{\textrm{FS}}) = (n+1)\ddbar \log (1+ |\underline{z}|^2) = (n+1) \omega_{\textrm{FS}}$.

\bigskip
\noindent
{\bf \ref{exerddbar}.} Adding a constant to $f$ we may assume that $f$ is positive. Applying Stokes' Theorem and Exercise \ref{extr},
$$0 \ge - \int_M f \ddbar f \wedge \omega^{n-1} = \frac{1}{n} \int_M | \partial f|^2 \omega^n,$$ so $\partial f=0$.

\noindent
{\bf \ref{ex22}.} For example, suppose that $\alpha>0$ and $\alpha<0$.  Then $\alpha$ contains a K\"ahler metric $\omega$ and $-\alpha$ contains a K\"ahler metric $\omega'$.  Then $[\omega+ \omega']=0$ so $\omega + \omega' = \ddbar f >0$ for some real-valued function $f$.  This contradicts Exercise \ref{exerddbar}.

\noindent
{\bf \ref{ex23}.} It is immediate from the definition that $\textrm{Ka}(M)$ is a convex cone.  For openness, let $\gamma_1, \ldots, \gamma_m$ be smooth closed $(1,1)$ forms with the property that $[\gamma_1], \ldots, [\gamma_m]$ is a basis for $H^{1,1}_{\ov{\partial}}(M, \mathbb{R})$.  If $\alpha$  in  $\textrm{Ka}(M)$ is represented by a K\"ahler metric $\omega$, then for $\ve_i>0$ sufficiently small $\omega + \sum_i \ve_i \gamma_i$ is K\"ahler.  Hence $[\alpha] + \sum_i \ve_i [\gamma_i]$ is in $\textrm{Ka}(M)$ for $\ve_i$ sufficiently small.

\noindent
{\bf \ref{prod}.} For (a) pick product coordinates.  (b) follows from (a).

\noindent
{\bf \ref{exerP1P1}.} Follows from Exercise \ref{prod}.

\noindent
{\bf \ref{ex26}.}  Let $\pi_E$ and $\pi_S$ be the projections onto $E$ and $S$. 
\begin{enumerate}
\item[(a)]  $c_1(M) = -[\pi^*_S\omega_S]$, and $\pi^*_S \omega_S \ge 0$, so $T=\infty$. 
\item[(b)]  $\omega(t) = \pi^*_E \omega_E + (1+t) \pi^*_S \omega_S$.  
\item[(c)]  The  torus fibers collapse and $\omega(t)/t$ converges to the K\"ahler-Einstein metric $\omega_S$ on $S$.
\end{enumerate}

\noindent
{\bf \ref{ex27}.} Observation:  $\alpha$ is nef if and only if for all $\ve>0$, we have $\alpha+\ve[\omega_0]>0$.
\begin{enumerate}  
\item[(a)] If $\alpha$ is nef then by the observation it is immediate that $\alpha$ is in the closure of $\textrm{Ka}(M)$.  Conversely, let $\alpha$ be in the closure of $\textrm{Ka}(M)$ so that there exist $\alpha_j \in \textrm{Ka}(M)$ with $\alpha_j \rightarrow \alpha$.  Let $\beta_1, \ldots, \beta_m$ be smooth closed $(1,1)$ forms so that the $[\beta_i]$ give a basis for $H^{1,1}_{\ov{\partial}}(M, \mathbb{R})$.  Then $\alpha - \alpha_j = \sum_i b_{i,j} [\beta_i]$ with $b_{i,j} \rightarrow 0$ as $j \rightarrow \infty$.  Now let $\ve>0$. For $j$ large enough, we have $ \sum_i b_{i,j} \beta_i\ge - \ve \omega_0$.  Let $\omega_j$ in $\alpha_j$ be K\"ahler.  Then $\omega_j + \sum b_{i,j} \beta_j \ge -\ve \omega_0$ and $\omega_j + \sum b_{i,j} \beta_j \in \alpha.$
\item[(b)] We just have to show that
$$\sup \{ t >0 \ | \ [\omega_0] - tc_1(M) \textrm{ is nef} \} \le T.$$
Suppose not.  Then $[\omega_0] - (T+\delta)c_1(M)$ is nef for some $\delta>0$ and so $(1+\ve)[\omega_0] - (T+\delta) c_1(M) >0$ for all $\ve>0$.  Hence
$$[\omega_0] - \frac{T+\delta}{1+\ve} c_1(M)>0,$$
a contradiction since we may choose $\ve>0$ so that $\displaystyle{\frac{T+\delta}{1+\ve} > T}$.

\end{enumerate}

\noindent
{\bf \ref{ex31}.}  (cf. \cite{TW2}).  Suppose $\omega=\omega(t)$ solves $\displaystyle{\ddt{} \omega = - \Ric{(\omega)}}$.  Then if we let $\varphi$ solve $\ddt \varphi = \log (\omega^n/\Omega)$ with $\varphi|_{t=0}=0$ we have 
$$\ddt{} (\omega - \hat{\omega}_t -\ddbar \varphi)=0,$$
which implies that $\omega = \hat{\omega}_t + \ddbar \varphi$, with $\varphi$ solving (\ref{ma}).

\noindent
{\bf \ref{ex45}.} For (a), 
suppose $f_t$ does not converge smoothly to $f$.  Then for some $k$ there exists $\ve>0$ and $t_i \rightarrow T_0$ such that $\| f_{t_i} - f \|_{C^k(M)} \ge \ve$ for all $i$.  Then applying Arzel\`a-Ascoli, after passing to a subsequence, $f_{t_i}$ converges smoothly to some function $\tilde{f}$ with $f \neq \tilde{f}$ and this contradicts the fact that $f_{t_i}$ converges to $f$ pointwise.
For (b), suppose $|\dot{f}_t|\le A$.  Then $f_t + tA$ is nondecreasing and bounded above so converges pointwise to a unique limit.  Now apply (a).

\noindent
{\bf \ref{ex32}.} Pick $\psi = \varphi - Bt$ for 
$$B = \inf_{M \times[0,S]} \log \frac{\hat{\omega}_t^n}{\Omega} -1.$$

\noindent
{\bf \ref{ex33}.} Put $Q= \dot{\varphi} - A \varphi$ for $A$ chosen so that $A\hat{\omega}_t \ge \chi$.  Then compute
$$\left( \ddt{} - \Delta \right) Q = \tr{\omega}{\chi} - A \dot{\varphi} + An - A \hat{\omega}_t \le -A\dot{\varphi} +An,$$
so that $\dot{\varphi}\le n$  at a maximum of $Q$ (if achieved at $t_0>0$).

\noindent
{\bf \ref{exeq}.} Pick coordinates at a point so that $g_0$ is the identity and $g$ is diagonal with eigenvalues $\lambda_1, \ldots, \lambda_n$.  Then, for example, if $\tr{\omega_0}{\omega} \le C_1$, we have $\sum_i \lambda_i \le C_1$ and so $\lambda_i \le C_1$.  On the other hand, we have $\lambda_1\lambda_2\cdots \lambda_n \ge C^{-1}$, so 
$$\frac{1}{\lambda_i} = \frac{\lambda_1 \cdots \widehat{\lambda_i} \cdots \lambda_n}{\lambda_1 \cdots \lambda_n} \le \frac{C_1^{n-1}}{C^{-1}}$$
where $\widehat{ \ }$ means ``omit''.
Hence $\tr{\omega}{\omega_0} = \sum_i \lambda_i^{-1} \le C_2 := nC_1^{n-1}C$.

\noindent
{\bf \ref{ex35}.}  Pick holomorphic coordinates at a point  with respect to $g_0$, so that $\nabla^0_k = \partial_k$.   Then
\[
\begin{split}
g_0^{i\ov{q}} g^{p\ov{j}} g^{k\ov{\ell}} B_{i\ov{j}k} \ov{B_{q\ov{p}\ell}} = {} & g_0^{i\ov{q}} g^{p\ov{j}} g^{k\ov{\ell}}  \left( \partial_i g_{k\ov{j}} - \frac{\partial_k (\tr{\omega_0}{\omega})}{\tr{\omega_0}{\omega}} g_{i\ov{j}} \right) \left( \partial_{\ov{q}} g_{p\ov{\ell}} - \frac{\partial_{\ov{\ell}} (\tr{\omega_0}{\omega})}{\tr{\omega_0}{\omega}} g_{p\ov{q}} \right) \\
= {} & (I) + (II) + (III),
\end{split}
\]
where, using the K\"ahler condition, we have
$$(I) = g_0^{i\ov{q}} g^{p\ov{j}} g^{k\ov{\ell}} \partial_i g_{k\ov{j}} \partial_{\ov{q}} g_{p\ov{\ell}} = g_0^{i\ov{q}} g^{p\ov{j}} g^{k\ov{\ell}} \nabla_k^0 g_{i\ov{j}} \nabla^0_{\ov{\ell}} g_{p\ov{q}},$$ 
and
\[
\begin{split}
(II) ={} & - 2\textrm{Re} \left( g_0^{i\ov{q}} g^{p\ov{j}} g^{k\ov{\ell}} \partial_i g_{k\ov{j}} \frac{\partial_{\ov{\ell}} (\tr{\omega_0}{\omega})}{\tr{\omega_0}{\omega}} g_{p\ov{q}} \right) \\
= {} & - 2 \textrm{Re} \left( g^{k\ov{\ell}} g_0^{i\ov{j}} \partial_k g_{i\ov{j}} \frac{\partial_{\ov{\ell}} (\tr{\omega_0}{\omega})}{\tr{\omega_0}{\omega}} \right) \\
= {} & - 2 \frac{ | \partial \tr{\omega_0}{\omega}|^2_g}{\tr{\omega_0}{\omega}},
\end{split}
\]
and
$$(III) = g_0^{i\ov{q}} g^{p\ov{j}} g^{k\ov{\ell}} \frac{\partial_k (\tr{\omega_0}{\omega})}{\tr{\omega_0}{\omega}} g_{i\ov{j}} \frac{\partial_{\ov{\ell}} (\tr{\omega_0}{\omega})}{\tr{\omega_0}{\omega}} g_{p\ov{q}} =  \frac{ | \partial \tr{\omega_0}{\omega}|^2_g}{\tr{\omega_0}{\omega}}.$$

\noindent
{\bf \ref{ex41}.} Both $-\Ric(\hat{\omega}_{\infty})$ and $\hat{\omega}_{\infty}$ lie  in $- c_1(M)$ so there exists $f$ with
$$\ddbar \log \hat{\omega}_{\infty}^n = \hat{\omega}_{\infty} + \ddbar f.$$
Set $\Omega= \hat{\omega}_{\infty}^n e^{-f+c}$ with $c$ chosen so that $\int_M \Omega = \int_M \omega_0^n$. 

\noindent
{\bf \ref{ex42}.} Similar to the proof in Lecture 3 (see Exercise \ref{ex31}) that (\ref{krf}) is equivalent to (\ref{ma}).

\noindent
{\bf \ref{ex43}.} At the maximum of $\varphi$ (if it occurs at $t_0>0$) we have $\ddbar \varphi \le 0$ and $\ddt{} \varphi \ge 0$ and hence from (\ref{nma}), $\varphi \le \log \frac{\hat{\omega}_t^n}{\Omega} \le C$.  The lower bound of $\varphi$ is similar.

\noindent
{\bf \ref{ex44}.} The only difference compared to the calculation of Lemma \ref{lemtr} is that in (\ref{calc1}) there is an extra term coming from the $-\omega$ in $\ddt{} \omega = - \Ric(\omega)-\omega$, which yields an additional $\displaystyle{\frac{-\tr{\omega_0}{\omega}}{\tr{\omega_0}{\omega}}}=-1.$

\noindent
{\bf \ref{ex46}.} For $\ve>0$, $(\ddt{} - \Delta) (f-\ve t) = -\ve<0$ and hence the maximum of $f-\ve t$ must occur at $t=0$ giving $f -\ve t \le \sup_M |f_0|$.  Let $\ve \rightarrow 0$.  The lower bound is similar.

\noindent
{\bf \ref{ex47}.} Consider $Q= \log \tr{\omega_0}{\omega}-A\varphi$ on $M \times [0,t]$ and show that, for $A$ sufficiently large,
$$\left( \ddt{} - \Delta \right) Q \le - \tr{\omega}{\omega_0} + C,$$
using the fact that $\dot{\varphi}$ is uniformly bounded.  If $Q$ achieves a maximum at $(x_0, t_0)$ with $t_0>0$ then since $\varphi$ is bounded we have $(\tr{\omega}{\omega_0})(x_0, t_0) \le C$ and so $(\tr{\omega_0}{\omega})(x_0, t_0)\le C'$.  Hence for any $(x,t)$, 
$$(\log \tr{\omega_0}{\omega})(x,t) - A\varphi(x,t) \le Q(x_0, t_0) \le \log C' - A \varphi(x_0,t_0)$$
and the claim follows after exponentiating.

\noindent
{\bf \ref{ex48}.} Since $\ddt{} \omega^n = \tr{\omega}{(\ddt{} \omega)} \, \omega^n= \Delta \dot{\varphi} \, \omega^n$ we have
$$\frac{d}{dt} P(t) = \int_M \Delta \dot{\varphi} \, \omega^n + \int_M \dot{\varphi} \Delta \dot{\varphi} \, \omega^n =  - \frac{1}{n} \int_M |\partial \dot{\varphi} |_g^2 \omega^n,$$
using Stokes' Theorem and Exercise \ref{extr}.

\noindent
{\bf \ref{ex49}.} Compute
$$\left( \ddt{} - \Delta \right) | \partial \dot{\varphi} |_g^2 = - | \nabla \nabla \dot{\varphi} |_g^2 - | \nabla \ov{\nabla} \dot{{\varphi}} |_g^2\le 0,$$
where $| \nabla \nabla \dot{\varphi} |_g^2 = g^{i\ov{j}} g^{k\ov{\ell}} \nabla_i \nabla_k \dot{\varphi} \nabla_{\ov{j}} \nabla_{\ov{\ell}} \dot{\varphi}$ etc.
Then 
$$\frac{d^2 P}{dt^2} \ge - \frac{1}{n} \int_M \Delta | \partial \dot{\varphi}|_g^2 \omega^n - \frac{1}{n} \int_M | \partial \dot{\varphi} |_g^2 \Delta \dot{\varphi} \, \omega^n \ge  - \frac{C}{n} \int_M | \partial \dot{\varphi}|^2_g \omega^n =  C\frac{dP}{dt},$$
since $\Delta \dot{\varphi}$ is uniformly bounded.    To show that $dP/dt \rightarrow 0$, we use the following elementary fact.  If $f : [0,\infty) \rightarrow \mathbb{R}$ satisfies the differential inequality $\dot{f} \ge Cf$ then if $f \ge -\ve$ at $t$ we have $f \ge -e^{2C} \ve$ on $[t,t+2]$ (consider $fe^{-Ct}$).  Then since $(dP/dt) (t_i) \rightarrow 0$ for $t_i \in [i,i+1]$ it follows that $dP/dt \rightarrow 0$.

\end{document}